\documentclass[11pt,a4paper]{amsart}
\usepackage{mathtools,amssymb,amsmath,amsopn,amsthm,graphicx,MnSymbol,centernot,stmaryrd,array}
\usepackage{tikz-cd}
\usepackage{enumitem}
\usepackage{amsaddr}
\usepackage{etoolbox}
\usepackage[all,curve,frame]{xy}
\patchcmd{\subsection}{-.5em}{.5em}{}{}
\usetikzlibrary{matrix}
\makeatletter
\newcommand*{\doublerightarrow}[2]{\mathrel{
  \settowidth{\@tempdima}{$\scriptstyle#1$}
  \settowidth{\@tempdimb}{$\scriptstyle#2$}
  \ifdim\@tempdimb>\@tempdima \@tempdima=\@tempdimb\fi
  \mathop{\vcenter{
    \offinterlineskip\ialign{\hbox to\dimexpr\@tempdima+1em{##}\cr
    \rightarrowfill\cr\noalign{\kern.5ex}
    \rightarrowfill\cr}}}\limits^{\!#1}_{\!#2}}}
\newcommand*{\triplerightarrow}[1]{\mathrel{
  \settowidth{\@tempdima}{$\scriptstyle#1$}
  \mathop{\vcenter{
    \offinterlineskip\ialign{\hbox to\dimexpr\@tempdima+1em{##}\cr
    \rightarrowfill\cr\noalign{\kern.5ex}
    \rightarrowfill\cr\noalign{\kern.5ex}
    \rightarrowfill\cr}}}\limits^{\!#1}}}
\newcommand{\nd}[1]{\begin{smallmatrix}#1\end{smallmatrix}}
\makeatother
\newcommand{\uwithtext}[1]{\DOTSB\mathbin{\text{\tikz{%
    \node (a) {#1};
    \draw[black, rounded corners=1.2ex] 
        ([yshift=-2pt]a.north west) -- (a.south west) -- (a.south east) -- ([yshift=-2pt]a.north east);}}}}

\begin{document}

\newtheorem{definition}{Definition}[section]
\newtheorem{definitions}[definition]{Definitions}
\newtheorem{deflem}[definition]{Definition and Lemma}
\newtheorem{lemma}[definition]{Lemma}
\newtheorem{proposition}[definition]{Proposition}
\newtheorem{theorem}[definition]{Theorem}
\newtheorem{corollary}[definition]{Corollary}
\newtheorem{algo}[definition]{Algorithm}
\theoremstyle{remark}
\newtheorem{rmk}[definition]{Remark}
\theoremstyle{remark}
\newtheorem{remarks}[definition]{Remarks}
\theoremstyle{remark}
\newtheorem{notation}[definition]{Notation}
\newtheorem{assumption}[definition]{Assumption}
\theoremstyle{remark}
\newtheorem{example}[definition]{Example}
\theoremstyle{remark}
\newtheorem{examples}[definition]{Examples}
\theoremstyle{remark}
\newtheorem{dgram}[definition]{Diagram}
\theoremstyle{remark}
\newtheorem{fact}[definition]{Fact}
\theoremstyle{remark}
\newtheorem{illust}[definition]{Illustration}
\theoremstyle{remark}
\newtheorem{que}[definition]{Question}
\theoremstyle{definition}
\newtheorem{conj}[definition]{Conjecture}
\newtheorem{scho}[definition]{Scholium}
\newtheorem{por}[definition]{Porism}
\DeclarePairedDelimiter\floor{\lfloor}{\rfloor}

\renewenvironment{proof}{\noindent {\bf{Proof.}}}{\hspace*{3mm}{$\Box$}{\vspace{9pt}}}
\author[Sardar, Gonzalez Chaio, and Trepode]{Shantanu Sardar, Alfredo Gonzalez Chaio, and Sonia Trepode}
\address{CEMIM, FCEyN \\Universidad Nacional de Mar del Plata\\ CONICET, Argentina}
\email{shantanusardar17@gmail.com,agonzalezchaio@gmail.com,\\ strepode@mdp.edu.ar}
\title{{On the Irreducible Morphisms for Skew group algebras}}
\subjclass[2020]{}

\begin{abstract}
For an algebraically closed field $K$, let $G$ be a finite abelian group of $K$-linear automorphisms of a finite-dimensional path algebra $KQ$ of a quiver $Q$. Under certain assumptions on the action of $G$, we show the existence of a certain kind of ``covering" that we call a Galois semi-covering functor, which becomes a Galois covering when the group action is free. We study the module category of its skew group algebra under this functor. As an application, we obtain a complete description of the irreducible morphisms and almost split sequences of skew group algebras and show that the (stable) rank is preserved under skewness. In particular, we determine the stable rank of skew-gentle algebras.
\end{abstract}

\maketitle
\section{Introduction}
Skew group algebras were first studied from the point of view of representation theory in \cite{ReRi85}. For a finite-dimensional algebra $\Lambda$ over a field $K$ and a finite group $G$ acting on $\Lambda$ by automorphisms, the skew group algebra $\Lambda G$ shares many representation-theoretic properties with $\Lambda$, often incarnated in properties of functors between $\mathrm{mod}\mbox{-}\Lambda$ and $\mathrm{mod}\mbox{-}\Lambda G$. Some central properties, like being a finite representation type, hereditary, Nakayama, or an Auslander algebra, and self-injectivity, are preserved under skewness. These algebras also share the same global dimensions. In contrast, a skew algebra of a basic (resp. connected) algebra is not always a basic (resp. connected) algebra. If $\Lambda$ is the quotient of a path algebra by an admissible ideal and $G$ is cyclic, then Reiten and Riedtmann describe the quiver $Q_G$ of (a basic version of) $\Lambda G$. Demonet \cite{Dem10} provides a complete description of the quiver of skew group algebra of a path algebra for arbitrary finite groups. Giovannini and Pasquali \cite{Gi19} consider a quiver with potential and construct the skew algebra of the associated Jacobian algebra for finite cyclic groups. The skew group algebra construction with a natural action of the dual group can recover the original quiver with potential. Moreover, it is shown that the property of being Frobenius and $d$-representation (in)finite is also preserved under skewness. The authors, with Plamondon \cite{GiPa19} study the skew group dg algebra of the Ginzburg dg algebra of a quiver with potential under the action of a finite abelian group and derive a functor between their cluster categories.

In this paper, we consider a finite abelian group that acts by permuting the vertices and preserving the arrow spans. Like all the articles described above, we follow the same convention that for a finite group $G$ of order $n$ and the algebra $\Lambda$, we study the skew algebra $\Lambda G$ when $n$ is invertible in $\Lambda$. De la Peña \cite{Jose83} studied skew group algebra for the algebras of finite representation type and observed that the skew group construction coincides with the Galois covering of an algebra in the case of admissible group action. This idea motivates us to introduce the concept of a Galois semi-covering functor (see Definition \ref{GSC}) between the module categories of an algebra and its skew algebra. We also show that this functor $F_\lambda$ is not a Galois covering when the group action has some fixed points. More precisely, we have the following result.

\begin{theorem} (Theorem \ref{diag})
Assume that $G$ acts on an algebra $\Lambda$ and $\Lambda G$ is the associated skew group algebra. Then for any $M, N\in \mathrm{mod}\mbox{-}\Lambda$, the functor $F_\lambda: \mathrm{mod}\mbox{-}\Lambda \to \mathrm{mod}\mbox{-}\Lambda G$ induces the following isomorphisms of vector spaces:
$$\mathrm{Hom}_{\Lambda G}(F_\lambda M,F_\lambda N)\approx \begin{cases}\bigoplus_{g\in G} \mathrm{Hom}_{\Lambda}(gM,N)&\mbox{ if } G_{M}\neq G \mbox{ or, }G_{N}= G;\\\bigoplus_{g\in G} \mathrm{Hom}_{\Lambda}(M,gN)&\mbox{ if }G_{N}\neq G \mbox{ or, }G_{M}= G.\end{cases}$$  
\end{theorem}

Classifying a finite-dimensional associative algebra as finite, tame, or wild is a challenging problem. According to Ringel, the computation of (stable) rank could make this task easier. In \cite{SVA23}, Srivastava, Sinha, and Kuber investigate the stable rank for special biserial algebras, in particular for gentle algebras. Geiss shows that skew gentle algebras are tame \cite{G} as they degenerate to gentle algebras, but the classification problem of arbitrary skew group algebras is still open. Considering this, we demonstrate that the functor $F_\lambda$ preserves the powers of the radical. As a result, rank and stable rank (see Definition \ref{ranstab}) are preserved under the skew group algebra construction, and we obtain the possible stable ranks for skew-gentle algebras. We have the following result.

\begin{theorem}(Theorem \ref{rank} and Theorem \ref{stable})
Let $G$ be an abelian group acting on an algebra $\Lambda$ and let $F_{\lambda}: \Lambda \to \Lambda G$ be a Galois semi-covering. Then 
\begin{enumerate}
\item $F_{\lambda}$ preserve powers of radicals.
\item rank of $\Lambda$ = rank of $\Lambda G$.
\item stable rank of $\Lambda$= stable rank of $\Lambda G$.
\end{enumerate}
\end{theorem}

A Galois semi-covering functor is not dense in general (see Example \ref{ARQSKW}). We introduce the concept of a semi-dense functor (see Definition \ref{SDF}) and show that $F_\lambda$ is semi-dense (see Corollary \ref{SDPF}). This allows us to describe the irreducible morphisms over the skew group algebra in terms of the irreducible morphisms over the given algebra. For most cases, irreducible morphisms of $\Lambda G$ can be obtained as the image of an irreducible morphism in $\Lambda$ under $F_\lambda$. Following from the previous result, we show when $F_\lambda$ preserves irreducibles. 

\begin{corollary}(Corollary \ref{irrsep})
Let $f:M\to N$ be an irreducible morphism in $\mathrm{mod}\mbox{-}\Lambda$ with $M,N \in \mbox{ind}(\Lambda)$. If $G_{M}\neq G$ or $G_{N}\neq G$, then $F_{\lambda}(f):F_{\lambda}M \to F_{\lambda}N$ is irreducible.
\end{corollary}

We also show how to recover the irreducible morphism of $\Lambda G$ when the above hypothesis doesn't hold.  More specifically, we establish the following result.

\begin{proposition} \label{irrstab} 
(Proposition \ref{irrstab}) Assume that $G$ acts on an algebra $\Lambda$ and $\Lambda G$ is the associated skew group algebra. Let $M,N \in \mbox{ind}(\Lambda G)$ such that $\hat{G}_{M}\neq \hat{G} $ and $\hat{G}_N \neq \hat{G}$. If $f \in  \mathrm{irr}_{\Lambda G}(M,N)$ then there exists a $f_1 \in  \mathrm{irr}_{\Lambda}(M_1,N_1)$ with $M_1,N_1 \in \mbox{ind}(\Lambda)$ such that ${}^gf$ for all $g\in G$ are the diagonal entries of the diagonal matrix $F_\lambda(f_1)$.
\end{proposition}

Finally, we describe the almost split sequences over the skew group algebra in terms of the almost split sequences over the given algebra.

\begin{theorem}(Theorem \ref{ARS})
Suppose $\mathcal{E}$ is an almost split sequence in $\Lambda$. Then the associated almost split sequence(s) in $\Lambda G$ have the following forms:
\begin{enumerate}
\item[$G_{\mathcal{E}}=G$] Here, $F_\lambda(\mathcal{E})=\uwithtext{Z}_{k=1}^n \bar{\mathcal{E}}^k$ where $\bar{\mathcal E}^k$ are the associated almost split sequences in $\Lambda G$ for each $k$ being glued via $Z$; In particular, if $Z=0$ then $F_\lambda(\mathcal{E})=\bigoplus_{k=1}^n \bar{\mathcal{E}}^k$.

\item[$G_{\mathcal{E}}\neq G$] Here, $F_\lambda(\mathcal{E})=\bar{\mathcal{E}}$ where $\bar{\mathcal E}$ is the associated almost split sequence in $\Lambda G$.
\end{enumerate}
\end{theorem}

The paper is organized as follows. In Section $2$, we present some basics about skew group algebra and introduce the concept of Galois semi-covering (Definition \ref{GSC}). We describe the quiver of a skew group algebra and end with Proposition \ref{semcov}, which demonstrates a Galois semi-covering functor $F$ between the bounded linear algebras $\Lambda$ and $\Lambda G$. In Section $3$, we develop the pushdown functor of $F$ in the module category and establish a Galois semi-covering functor $F_\lambda$ from $\mathrm{mod}\mbox{-}\Lambda$ to $\mathrm{mod}\mbox{-}\Lambda G$ (Theorem \ref{diag}). Proposition \ref{rem} states that indecomposable modules are not necessarily preserved under a Galois semi-covering functor, but this preservation holds under some hypotheses given in Proposition \ref{Ind}. Since a Galois semi-covering functor does not hold the dense property (see Example \ref{ARQSKW}), we introduce the semi-dense property of a functor at the end of this section (see Definition \ref{SDF}) and show that $F_\lambda$ is semi-dense (Corollary \ref{SDPF}), which is essential to describe the indecomposable modules in $\mathrm{mod}\mbox{-}\Lambda G$. In Section $4$, We show that the (stable) rank is preserved under the skew group algebra construction (see Theorems \ref{rank} and \ref{stable}) to determine the stable rank of a skew gentle algebra $\Lambda$ with at least one band when the characteristic of $K$ is different from $2$ and that lies in $\{\omega,\omega+1,\omega+2\}$ (see Corollary \ref{SRA}). In section $5$, we investigate the irreducible morphisms of a skew group algebra (Theorem \ref{irrars}) and obtain a complete description by applying the semi-dense property of $F_\lambda$. In contrast to a Galois covering, we observe that a Galois semi-covering does not preserve the irreducible morphisms (Example \ref{irrnp}). We end the section describing the almost split sequences of $\Lambda G$ (see Theorem \ref{ARS}).

\subsection*{Acknowledgements}
The authors were supported by PICT 2021-01154 ANPCyT. The first author had a stay at the CEMIM-Universidad Nacional de Mar del Plata, supported by a postdoctoral scholarship from CONICET.

\section{Galois semi-covering between an algebra and its skew group algebra}
In this section, we introduce a Galois semi-covering functor between two linear categories and establish such a functor between an algebra and its skew group algebra.

\subsection{Galois semi-covering functor in a linear category}
\begin{definition}[Skew group algebras]
Let $G$ be a finite group acting on an algebra $\Lambda$ by automorphisms. The skew group algebra $\Lambda G$ is the algebra defined by:
\begin{enumerate}
    \item its underlying vector space is $\Lambda \otimes_K KG$;
    \item multiplication is given by $(\lambda \otimes g)(\mu \otimes h) = \lambda g(\mu) \otimes gh$ for $\lambda,\mu\in\Lambda$ and $g,h\in G$, extended by linearity and distributivity.
\end{enumerate}
\end{definition}

There is a natural algebra monomorphism $\Lambda\hookrightarrow\Lambda G$ given by $\lambda \to \lambda\otimes1$. Notice that the algebra $\Lambda G$ is not basic in general.


\begin{notation}
Assume that $Q=(Q_0,Q_1,s,t)$ is a finite quiver, where $Q_0$ is the set of vertices, $Q_1$ is the set of arrows and $s,t: Q_1\to Q_0$ denote the source and target functions. Fix $\tilde I$ to be a set of representatives of $Q_0$ under the action of G. This choice affects the rest of the paper as it corresponds to choosing an idempotent subalgebra of $(KQ)G$ which is Morita equivalent to $(KQ)G$. We denote the elements of $\tilde I$, for instance, $i_0$ and $j_0$. Set $G_{i_0}$ as the stabilizer of $i_0$ in $G$ and $G_{i_0j_0}:=G_{i_0}\cap G_{j_0}$. The space generated by the arrows from vertex $i$ to vertex $j$ is denoted by $\Lambda_1(i,j)$. We write $\Lambda_1$ for the arrow space of $Q$ and $\Lambda_1G\subseteq \mathcal (KQ)G$ for the space generated by elements of the form $\lambda \otimes g$ with $\lambda \in \Lambda_1$ and $g \in G$. Denote the $G$-orbit of a path $p$ by $O_p$ and define $\tilde {Q_0}':=\{i \in Q_0\mid |O_{i}|=|G|\}$. Moreover, $\mbox{ind}(\Lambda)$ denotes the isomorphism classes of indecomposable modules over $\Lambda$, whereas by $\mathrm{irr}_{\Lambda}(M,N)$, we mean the collection of all irreducible morphisms between two $\Lambda$-modules $M$ and $N$.
\end{notation}

\begin{assumption}
Let us fix a finite abelian group $G$ of order $n$ acting on $KQ$ such that $G_{i_0}\neq G$ implies $|O_{i_0}|=n$ for each $i_0\in \tilde I$. Recall that the set of all irreducible representations of $G$, denoted $\hat{G}$, forms a group w.r.t. tensor product of representations with $|\hat{G}|=n$.
\end{assumption}

The following remark is essential to figure out the arrow set for skew group algebras.
\begin{rmk}\label{arrowpreservation} Since $G$ acts by automorphisms, the action preserves the natural grading on $KQ$ by the length of paths. Suppose $i,j \in Q_0$.
\begin{enumerate}
\item Let $V$ be the $G$-orbit of $\Lambda_1(i,j)$. Then, by \cite[Lemma~3.1]{GiPa19}, there is a basis for $V$ such that $G$ maps arrows in $V$ to multiples of arrows, more specifically, each arrow $a:i\to j$ corresponds $\chi_a \in \hat{G}_{ij}$ such that for every $g\in G_{ij}$ we have $g(a)= \chi_a(g)a$. We can repeat to generalize this for every orbit in $Q_0 \times Q_0$ which is weaker than our initial assumption that $G$ preserves the arrowspans;

\item If $i$ is a source (or a sink), then $g(i)$ is also a source (or a sink);

\item Indegree (or outdegree) of $i$ is equal to the indegree (or outdegree) of $g(i)$.
\end{enumerate}
\end{rmk} 

In contrast to a Galois covering of a linear category \cite{BoGa81}, the action of the group $G$ is not necessarily free in skew group algebras. The following remark states that these two constructions are essentially the same in the case of a free $G$-action. 

\begin{theorem}\label{Jose}\cite[Corollary~5.3]{Jose83}
Suppose a group $G$ acts on an algebra $\Lambda$ of finite representation type and $\Lambda G$ is the associated skew group algebra. If $G$ acts freely and $\tilde{F}:\tilde{\Lambda} \to \Lambda$ is a Galois $G$-covering then $\Lambda G$ is Morita equivalent to $\tilde{\Lambda}$.
\end{theorem}

We verified Theorem \ref{Jose} for infinite representation type algebras as well, which motivates us to introduce the Galois semi-covering to explain skewness. 

\begin{definition}[Galois semi-covering functor]\label{GSC}
Let $A,B$ be linear categories with $G$ a group acting on $A$. A functor $F: A\to B$ is called a Galois semi-covering functor if for any $X,Y\in \mathrm{Ob}(A)$, the following hold:
$$B(FX,FY)\approx \begin{cases}\bigoplus_{g\in G} A(gX,Y)&\mbox{if } G_{X}\neq G \mbox{ or, } G_{Y}= G;\\\bigoplus_{g\in G} A(X,gY)&\mbox{if }G_{Y}\neq G \mbox{ or, } G_{X}= G.\end{cases}$$  
\end{definition}

\begin{rmk}
Note that if a functor $F: A\to B$ between two linear categories $A$ and $B$ satisfy the following assertion: 
$$B(FX,FY)\approx \begin{cases}\bigoplus_{g\in G} A(gX,Y)&\mbox{if } G_{X}\neq G;\\\bigoplus_{g\in G} A(X,gY)&\mbox{if }G_{Y}\neq G;\\A^{|G|}(X,Y)&\mbox{ if } G_{XY}= G,\end{cases}$$  
for any $X,Y\in \mathrm{Ob}(A)$, where $A^{|G|}(X,Y)=\bigoplus_{g\in G} A(X,Y)$, then it's a Galois semi-covering functor. Clearly if $G_{XY}=G$, then $B(FX,FY)\approx A^{|G|}(X,Y)\approx \bigoplus_{g\in G} A(gX,Y)\approx \bigoplus_{g\in G} A(X,gY)$. Throughout this paper, we follow this equivalent definition for the Galois semi-covering.
\end{rmk}

\subsection{Galois semi-covering functor between $\Lambda$ and $\Lambda G$}
We describe an idempotent $\bar e\in \Lambda G$ following \cite{ReRi85} such that $\bar e(\Lambda G)\bar e$ is basic and Morita equivalent to $\Lambda G$. We decompose $\bar e$ as a sum of primitive orthogonal idempotents to label the vertices of $Q_G$ and elements in $\bar e(\Lambda G)\bar e$ chosen to be the arrows. 

\noindent{Vertex Set $(Q_{G_0})$:}
The vertices of $Q_G$ (\cite{Dem10}) are given by
$$Q_{G_0}=\{(i_0, \rho)\mid i_0 \in \tilde I, \rho \in \hat{G}_{i_0}\}.$$
The idempotent of $(KQ)G$ corresponding to the vertex $(i_0, \rho)$ is $$e_{i_0\rho} = i_0\otimes e_\rho \mbox{, where } e_\rho = \frac{1}{|G_{i_0}|} \sum_{g\in G_{i_0}} \rho(g)g$$ is an idempotent of $KG_{i_0}$. Write $\tilde I = \tilde I'\sqcup \tilde I''$ where $\tilde I':=\tilde I\cap \tilde {Q_0}'$.

For each vertex $i_0\in \tilde I'$, $\hat{G}_{i_0} (:=\{tr\})$ is trivial, and hence, $e_{i_0}:= e_{i_0 tr} = i_0\otimes 1$ is the associated idempotent in $\Lambda G$. Consider the idempotent of $KQ_G$
$$\bar e =
\sum_{i_0\in \tilde I} \bar {e_{i_0}} \mbox{ where } \bar {e_{i_0}} = \sum_{\rho \in \hat{G}_{i_0}}e_{i_0\rho}.$$

\begin{rmk}
The idempotent $\bar e$ is such that $\bar e(KQ)G\bar e$ is basic and Morita equivalent to $(KQ)G$. Moreover, there is an explicit isomorphism $KQ_G \to \bar e(KQ)G\bar e$, see \cite{GiPa19}. Although the construction of $Q_G$ depends on the choice of $\tilde I$ (equivalently of $\bar e$), it will lead to isomorphic quivers.
\end{rmk}

\noindent{Arrow Set $(Q_{G_1})$:}
Following \cite{GiPa19}, here we fix the arrow set $Q_{G_1}$. For each $i\in Q_0$, choose an element $\kappa_i \in G$ such that $\kappa_i(i) \in \tilde I$. Fix $\kappa_{i_0} = 1$ for each $i_0 \in \tilde I$. For each $i_0, j_0 \in \tilde I$, choose a set $R_{i_0j_0}$ of representatives of $O_{i_0}$ under the action of $G_{j_0}$ and set
$$D(i_0, j_0) = \{a : i \to j_0 \in \mathcal Q_1, i \in R_{i_0j_0}\}.$$

The set of arrows in $Q_G$ from $(i_0, \rho)$ to $(j_0, \sigma)$ is in bijection with the set
$$\{a \in D(i_0, j_0) \mid \rho|G_{i_0j_0} = \sigma|G_{i_0j_0}\chi_a \}.$$   

\begin{example}\label{SKA}
Here, we compute the skew group algebra $\bar{\Lambda}:=\Lambda\mathbb{Z}_2$ of an algebra $\Lambda$ where the action of $\mathbb{Z}_2:=\{e,\tau\}$ on $\Lambda$ is given by $\tau(v_i)=v'_i$ for $i=1,2$ and $\tau(v_i)=v_i$ for $i=3,4$. Indeed, $\Lambda$ is also Morita equivalent to $\bar{\Lambda}\mathbb{Z}_2$ if we consider the action $\tau(\bar v_i)=\bar v_i$ for $i=1,2$ and $\tau(\bar v_i)=\bar v'_i$ for $i=3,4$.
\begin{figure}[h]
\begin{minipage}[b]{0.45\linewidth}
\centering
\begin{tikzcd}[sep={3.4em,between origins}]
v_1 \arrow[r, "\alpha"] \arrow[rdd, "\beta'"', shift right]  & v_2 \arrow[rd, "\gamma"]   &                         &     \\
                                                             &                            & v_3 \arrow[r, "\delta"] & v_4 \\
v'_1 \arrow[r, "\alpha'"] \arrow[ruu, "\beta"', shift right] & v'_2 \arrow[ru, "\gamma'"] &                         &    
\end{tikzcd}
\caption{$\Lambda$ with $\rho= \{\gamma\alpha+\gamma'\beta, \gamma\beta'+\gamma'\alpha'\}$}
\label{fig:my_label}
\end{minipage}
\hspace{0.45cm}
\begin{minipage}[b]{0.5\linewidth}
\centering
\begin{tikzcd}[sep={3.4em,between origins}]
                                                                                      &                                                              & \bar v_3 \arrow[r, "\bar\delta"]   & \bar v_4  \\
\bar v_1 \arrow[r, "\bar\beta"', bend right=49] \arrow[r, "\bar\alpha", bend left=49] & \bar v_2 \arrow[ru, "\bar\gamma"] \arrow[rd, "\bar\gamma'"'] &                                    &           \\
                                                                                      &                                                              & \bar v'_3 \arrow[r, "\bar\delta'"] & \bar v'_4
\end{tikzcd}
    \caption{$\Lambda'$ with $\rho= \{\bar\gamma'\bar\alpha, \bar\gamma\bar\beta\}$}
    \label{fig:my_label}
    \end{minipage}
\end{figure}
\end{example}

The following remark deals with the interplay between the arrow space of the underlying quivers of an algebra and its skew group algebra.

\begin{rmk}\label{arrowcorrespondence}
Here, we fix a basis of $Q_{G_1}$. Let $\beta\in Q_1$. There are four different cases.

\begin{enumerate}
\item If $s(\beta), t(\beta)\in \tilde {Q_0}'$, then there is exactly one arrow $\alpha$ in $O_\beta$ of the form $\alpha: g^ti_0 \to j_0$, with $i_0, j_0 \in \tilde I'$ and $0 \leq t \leq n-1$. Define $\tilde \alpha \in \bar e\Lambda G\bar e$ by $\tilde \alpha := \alpha \otimes g^t: e_{i_0} \to e_{j_0} \in Q_{G_1}$. 

\item If $s(\beta)\in \tilde {Q_0}'$ but $t(\beta)\notin \tilde {Q_0}'$, then there is exactly one arrow $\alpha$ in $O_\beta$ of the form $\alpha: i_0 \to j_0$, with $i_0 \in \tilde I', j_0 \in \tilde I''$. Then define $\tilde \alpha^\sigma \in \bar e\Lambda G\bar e$ by $\tilde \alpha^\sigma:= (1 \otimes e_\sigma)(\alpha \otimes 1):e_{i_0}\to e_{j_0 \sigma}\in Q_{G_1}$ for $\sigma\in \hat{G}_{j_0}$.

\item If $s(\beta)\notin \tilde {Q_0}'$ but $t(\beta)\in \tilde {Q_0}'$, then there is exactly one arrow $\alpha$ in $O_\beta$ of the form $\alpha: i_0 \to j_0$, with $i_0 \in \tilde I'', j_0 \in \tilde I'$. Then define $\tilde \alpha^\rho \in \bar e\Lambda G\bar e$ by $\tilde \alpha^\rho:= \alpha \otimes e_\rho:e_{i_0 \rho}\to e_{j_0}\in Q_{G_1}$ for $\rho\in \hat{G}_{i_0}$.

\item If $s(\beta), t(\beta)\notin \tilde {Q_0}'$, then $\beta$ is of the form $\beta: i_0 \to j_0$, with $i_0, j_0 \in \tilde I''$. By Remark \ref{arrowpreservation}, we have $g(\beta)= \chi_\beta(g) \beta$. Then define $\tilde \beta^\rho \in \bar e\Lambda G\bar e$ by $\tilde \beta^\rho:= \beta \otimes e_\rho:e_{i_0 \rho}\to e_{j_0 \rho\chi^{-1}_a(g)}\in Q_{G_1}$ for $\rho\in \hat{G}_{i_0}$.
\end{enumerate}
\end{rmk}

The following example explains different cases of the above remark.
\begin{example}\label{repcl}
Consider the algebras $\Lambda$ and $\bar{\Lambda}$ from Example \ref{SKA}. Let us fix $\tilde I_\Lambda'=\{v_1,v'_2\}$, $\tilde I_\Lambda''=\{v_3,v_4\}$, $\tilde I'_{\bar{\Lambda}}=\{\bar v'_3,\bar v_4\}$, $\tilde I''_{\bar{\Lambda}}=\{\bar v_1,\bar v_2\}$.

\noindent{\textbf{Case 1:}} Here, $\beta, \beta', \alpha, \alpha'$ in $\Lambda$ belong to this case. But for $\beta$ and $\beta'$, the only representative in $O_\beta$ is $v_1\xrightarrow{\beta} v'_2$ itself, whereas for $\alpha$ and $\alpha'$, the only representative in $O_\alpha$ is $\tau v_1=v'_1\xrightarrow{\alpha'} v'_2$. So, we consider the arrows $e_{v_1}=\bar v_1 \doublerightarrow{\tilde\alpha'=\bar \alpha}{\tilde \beta=\bar\beta} \bar v_2= e_{v_2}$ in $Q_{\bar{\Lambda}}$.

\noindent{\textbf{Case 2:}} Here, $\gamma, \gamma'$ in $\Lambda$ belong to this case. The only representative in $O_\gamma$ is $v'_2\xrightarrow{\gamma'} v_3$. So we consider the arrows $e_{v'_2}=\bar v_2 \xrightarrow{\tilde \gamma'^e=\bar\gamma} \bar v_3=e_{v_3 e}$ and $e_{v'_2}=\bar v_2 \xrightarrow{\tilde\gamma'^\tau=\bar\gamma'} \bar v'_3=e_{v_3 \tau}$ in $Q_{\bar{\Lambda}}$.

\noindent{\textbf{Case 3:}} Here, $\bar\gamma, \bar\gamma'$ in $\bar{\Lambda}$ belong to this case. The only representative is $\bar v_2\xrightarrow{\bar\gamma'} \bar v'_3$ in $O_{\bar\gamma}$. We consider the arrows $e_{\bar v_2 e}= v_2 \xrightarrow{\tilde {\bar\gamma'}^e=\gamma} v_3=e_{\bar v_3}$ and $e_{\bar v_2 \tau}=v'_2 \xrightarrow{\tilde{\bar \gamma'}^\tau=\gamma'} v_3=e_{\bar v_3}$ in $Q_{\Lambda}$.

\noindent{\textbf{Case 4:}} Here, $\bar\alpha, \bar\beta$ in $\bar{\Lambda}$ belong to this case. Let us fix $\chi_{\bar\alpha}=e,\chi_{\bar\beta}=\tau$. Then we have the arrows $e_{\bar v_1 e}=v_1 \xrightarrow{\tilde{\bar \alpha}^e=\alpha} v_2=e_{\bar v_2 e}$, $e_{\bar v_1 \tau}=v'_1 \xrightarrow{\tilde{\bar \alpha}^\tau=\alpha'} v'_2=e_{\bar v_2 \tau}$, $e_{\bar v_1 e}=v_1 \xrightarrow{\tilde{\bar \beta}^e=\beta} v'_2= e_{\bar v_2 \tau}$, $e_{\bar v_1 \tau}=v'_1 \xrightarrow{\tilde{\bar \beta}^\tau=\beta'} v_2=e_{\bar v_2 e}$ in $Q_{\Lambda}$.
\end{example}

Below, we demonstrate a Galois semi-covering functor between an algebra and its skew group algebra. Later, we consider its pushdown functor, which becomes the Galois semi-covering functor between their module categories.

Consider a functor $F: KQ \to (KQ)G$ by setting for each $i\in Q_0$, $$F(i)=\bar {e_{i_0}} \mbox{ where, } i_0\in O_i\cap \tilde I.$$

Note that $F(\alpha)$ is determined by its endpoints for each $\alpha\in Q_1$. The next proposition, which follows from Remark \ref{arrowcorrespondence}, ensures the existence of a Galois semi-covering.  
\begin{proposition}\label{semcov}
$\Lambda_1G(F(i), F(j))\approx \begin{cases}\bigoplus_{g\in G} \Lambda_1(gi,j)&\mbox{ if } G_{i}\neq G;\\\bigoplus_{g\in G} \Lambda_1(i,gj)&\mbox{if }G_{j}\neq G;\\\Lambda_1^{|G|}(i,j)&\mbox{ if } G_{ij}= G.\end{cases}$   

In particular, $F$ produces a Galois semi-covering functor between $\Lambda_1$ and $\Lambda_1 G$. 
\end{proposition}

\begin{rmk}
Note that the Galois semi-covering functor between $\Lambda_1$ and $\Lambda_1 G$ can be extended to a Galois semi-covering functor between the bounded algebras.   
\end{rmk}

\begin{example}
Consider the algebras $\Lambda$ and $\bar{\Lambda}$ from Example \ref{SKA}. Here, $G_{v_1}, G_{v_2}\neq G$, $F(v_1)=\bar {e_{v_1}}=\bar v_1$ and $F(v_2)=\bar {e_{v'_2}}=\bar v_2$. Clearly, $\bar{\Lambda}_1(F(v_1), F(v_2))$ and $\bigoplus_{g\in G} \Lambda_1(gv_1,v_2)$ both have dimensions $2$. On the other hand, $G_{\bar v_1 \bar v_2}=G$, $F(\bar v_1)={e_{v_1}}$ and $F(\bar v_2)= {e_{v'_2}}$. Here, $\Lambda_1(F(\bar v_1), F(\bar v_2))$ and $\bar{\Lambda}_1(\bar v_1,\bar v_2)$ have dimensions $4$ and $2$ respectively.
\end{example}

\begin{section}{Galois semi-covering functor in the module category}
In this section, we introduce a Galois semi-covering functor between the module category of $\Lambda$ and $\Lambda G$, which appears as a pushdown functor of a Galois semi-covering between these bounded algebras. 


\begin{definition}\label{dgsc}
Let $F: \Lambda \to \Lambda G$ be the Galois semi-covering functor. We define the pushdown functor $F_\lambda: \mathrm{mod}\mbox{-}\Lambda \to \mathrm{mod}\mbox{-}\Lambda G$ as follows:

Suppose $M \in \mathrm{mod}\mbox{-}\Lambda$, then $F_\lambda M= \bigoplus_{i_0\in \tilde I} F_\lambda M(\bar {e_{i_0}})$ where, for each $\bar {e_{i_0}}\in \Lambda G$, we set
$$F_\lambda M(\bar {e_{i_0}}):= \bigoplus_{F(x)=\bar {e_{i_0}}}M(x).$$

Assume that $\bar{\alpha}\in \Lambda_1G(\bar {e_{i_0}},\bar {e_{j_0}})$. Consider the following cases:
\begin{enumerate}
\item If $G_{i_0}\neq G$ then we have $\bar {e_{i_0}}=e_{i_0}$. By Proposition \ref{semcov}, $F$ induces an isomorphism $\Lambda_1G(F{i_0},F{j_0})\approx \bigoplus_{g\in G} \Lambda_1(g{i_0},{j_0})$ and hence, there is an arrow $\alpha_h: h{i_0}\to {j_0}$ for some $h\in G$ such that $\bar{\alpha}=F(\alpha_h)$. Then the homomorphism $F_\lambda M(\bar{\alpha}): F_\lambda M(e_{i_0}) \to F_\lambda M(\bar {e_{j_0}})$ is defined by homomorphism:
$$(\mu_g) \mapsto (\sum_{g\in G} M(g\alpha_h)(\mu_g)).$$

\item If $G_{i_0}=G$ but $G_{j_0}\neq G$ then we have $\bar {e_{j_0}}=e_{j_0}$. By Proposition \ref{semcov}, $F$ induces an isomorphism $\Lambda_1G(Fi_0,Fj_0)\approx \bigoplus_{g\in G} \Lambda_1(i_0,gj_0)$ and hence there is an arrow $\alpha_h: i_0\to hj_0$ for some $h\in G$ such that $\bar{\alpha}=F(\alpha_h)$. Then the homomorphism $F_\lambda M(\bar{\alpha}): F_\lambda M(\bar {e_{i_0}}) \to F_\lambda M(e_{j_0})$ is defined by homomorphism:
$$\mu \mapsto (M(g_1\alpha_h)(\mu),\hdots, M(g_n\alpha_h)(\mu)).$$

\item If $G_{i_0j_0}= G$ then by Proposition \ref{semcov}, $F$ induces an isomorphism, $\Lambda_1G(Fi_0,Fj_0)\approx \Lambda^{|G|}_1(i_0,j_0)$ and hence there is an arrow $\alpha: i_0\to j_0$ such that $\bar{\alpha}=F(\alpha)$. Then the homomorphism $F_\lambda M(\bar{\alpha}): F_\lambda M(\bar {e_{i_0}}) \to F_\lambda M(\bar {e_{j_0}})$ is defined by homomorphism:
$$\mu \mapsto M(\alpha)(\mu).$$
\end{enumerate}
\end{definition}

The following example illustrates the push-down functor $F_\lambda:\mathrm{mod}\mbox{-}\Lambda \to \mathrm{mod}\mbox{-}\Lambda G$.

\begin{example}
Consider the algebras $\Lambda$ and $\bar{\Lambda}$ from Example \ref{SKA}. Here, for a representation $M$ of $\Lambda$ we compute $F_\lambda M$ of $\bar{\Lambda}$ using the above definition as follows: $$F_\lambda M(\bar {e_{v_2}})= F_\lambda M(\bar{v_2})= M(v_2)\oplus M(v'_2)=K \oplus K.$$
$$F_\lambda M(\bar {e_{v_3}})= M(v_3)=K \oplus K \mbox{ i.e. } M(\bar {v_3})=M(\bar {v_3}')=K \oplus K.$$
\begin{figure}[h]
\begin{minipage}[b]{0.45\linewidth}
\centering
\begin{tikzcd}[sep={3.9em,between origins}]
K \arrow[r, "1"] \arrow[rdd, "0"', shift right] & K \arrow[rd, "\left(\begin{matrix} 1 \\ 1  \end{matrix}\right)"]  &                                                                            &     \\
                                                &                                                                   & K^2 \arrow[r, "\left( \begin{matrix} 1 \phantom{0} 0 \\ 0 \phantom{0} 1  \end{matrix}\right)"] & K^2 \\
0 \arrow[r, "0"] \arrow[ruu, "1"', shift right] & K \arrow[ru, "\left(\begin{matrix} -1 \\ -1  \end{matrix}\right)"'] &                                                                            &    
\end{tikzcd}

\caption{$M$}
\label{fig:my_label}
\end{minipage}
\hspace{0.45cm}
\begin{minipage}[b]{0.5\linewidth}
\centering
\begin{tikzcd}[sep={3.9em,between origins}]
                                                                                                                                                            &                                                                                                                                & K^2 \arrow[r, "\left(\begin{matrix} 1 \phantom{0} 0 \\ 0 \phantom{0} 1  \end{matrix}\right)"] & K^2 \\
K \arrow[r, "\left(\begin{matrix} 0 \\ 1  \end{matrix}\right)"', bend right=49] \arrow[r, "\left(\begin{matrix} 1 \\ 0  \end{matrix}\right)", bend left=49] & K^2 \arrow[ru, "\left(\begin{matrix} 1 \phantom{0} 0\\ 1 \phantom{0} 0\end{matrix}\right)"] \arrow[rd, "\left(\begin{matrix} 0 \phantom{0} -1 \\ 0 \phantom{0} -1\end{matrix}\right)"'] &                                                                           &     \\
                                                                                                                                                            &                                                                                                                                & K^2 \arrow[r, "\left(\begin{matrix} 1 \phantom{0} 0 \\ 0 \phantom{0} 1  \end{matrix}\right)"] & K^2
\end{tikzcd}    
\caption{$F_\lambda(M)$}
    \label{fig:my_label}
    \end{minipage}
\end{figure}

We compute $F_\lambda M(\bar{\alpha})$ for $\bar{\alpha}:\bar{v}_1\to \bar{v}_2\in Q_{G_1}$. Here, $\bar{\alpha}\in \Lambda_1G(\bar {e_{v_1}},\bar {e_{v_2}})$ where, $G_{v_1}, G_{v_2}\neq G$ (see Example \ref{repcl}). So we use the first case in Definition \ref{dgsc}. By Proposition \ref{semcov}, consider the arrow $\alpha_h=\alpha: v_1\to v_2$  such that $\bar{\alpha}=F(\alpha)$. In $M$, $M(\alpha_h)=1: K \to K$ and $M(\tau\alpha_h)=0: 0\to K$. Therefore, $F_\lambda M(\bar{\alpha}): F_\lambda M(e_{v_1}) \to F_\lambda M(v_2)$ is defined by: $\left(\begin{matrix} 1 \\ 0  \end{matrix}\right):K\bigoplus 0 \to K\bigoplus K.$ 

Now we compute $F_\lambda M(\bar{\gamma})$ for $\bar{\gamma}:\bar{v}_2\to \bar{v}_3\in Q_{G_1}$. Here, $\bar{\gamma}\in \Lambda_1G(e_{v_2},e_{v_3e})$ where, $G_{v_2}\neq G, G_{v_3}=G$ (see Example \ref{repcl}). Thus, we use the first case in Definition \ref{dgsc}. Denote $(\bar{\gamma}, \bar{\gamma'})$ by $\tilde{\gamma}$. By Proposition \ref{semcov}, consider the arrow $\alpha_h=\gamma: v_2\to v_3$ such that $(\tilde{\gamma})=F(\gamma)$. In the representation $M$, $M(\alpha_h)=\left(\begin{matrix} 1 \\ 1  \end{matrix}\right): K \to K^2$ and $M(\tau\alpha_h)=\left(\begin{matrix} -1 \\ -1  \end{matrix}\right): K\to K^2$. Hence, $F_\lambda M(\tilde{\gamma}): F_\lambda M(e_{v_2}) \to F_\lambda M(\bar{e_{v_3}})$ is defined by $(\left(\begin{matrix} 1 \phantom{0} 0\\ 1 \phantom{0} 0 \end{matrix}\right), \left(\begin{matrix} 0 \phantom{0} -1\\ 0 \phantom{0} -1 \end{matrix}\right)):K^2 \doublerightarrow{(F_\lambda M(\bar{\gamma}),0)}{(0, F_\lambda M(\bar{\gamma'}))} K^2\bigoplus K^2.$ 
\end{example}

\noindent{\textbf{Morphism in $\mathrm{mod}\mbox{-}\Lambda G$:}} Suppose $f:=(f_{i_0})_{i_0\in Q_0}: M\to N$ is a homomorphism in $\mathrm{mod}\mbox{-}\Lambda$ where, $f_{i_0}: M(i_0) \to N(i_0)$. Then $F_\lambda(f):= (\hat{f}_{\bar{e_{i_0}}}): F_\lambda M \to F_\lambda N$ where, $\hat{f}_{\bar{e_{i_0}}}: F_\lambda M(\bar{e_{i_0}}) \to F_\lambda N(\bar{e_{i_0}})$ is defined by homomorphisms $f_{i}: M(i) \to N(i)$, for all $i\in \mathcal{O}(i_0)$.

The next theorem states the involutive effect of skew group algebra construction.
\begin{theorem}\label{adj}\cite[Corollary~5.2]{ReRi85}
For an abelian group $G$, the algebra $\Lambda$ is Morita equivalent with the skew group algebra of $(\Lambda G)\hat{G}$, where the action of $\hat{G}$ on $\Lambda G$ is defined by $\chi (\lambda \otimes g): = \chi(g) \lambda \otimes g$ for $\lambda \in \Lambda, g \in G$.  
\end{theorem}


Given a $M\in \mathrm{mod}\mbox{-}\Lambda$ we define the module ${}^gM$ where ${}^gM(x):=M(gx)$ and a module homomorphism $f: M \to N$ we denote by ${}^gf$ the $\Lambda$-module homomorphism ${}^gM \to {}^gN$ such that ${}^gf(x):= f(gx)$, for any $x \in Q_0$. This defines an action of $G$ on $\mathrm{mod}\mbox{-}\Lambda$. Moreover, the map $f\to {}^gf$ defines isomorphism of vector spaces $\Lambda(M,N) \approx \Lambda({}^gM, {}^gN)$. Denote by $G_{M}$, the stabilizer of $M$ in $G$ and $G_{MN}=G_{M}\cap G_{N}$. 

Here are some observations about the push-down functor.

\begin{rmk}\label{adja}
Theorem \ref{adj} ensures the existence of a Galois semi-covering $G_\lambda: \mathrm{mod}\mbox{-}\Lambda G \to \mathrm{mod}\mbox{-}\Lambda$. Moreover, note that, $G_\lambda (F_\lambda M) (x)= F_\lambda (Fx)= \bigoplus_{g\in G} M({}^gx)= \bigoplus_{g\in G} {}^gM(x)$ for each $x\in \Lambda$. This implies that $G_\lambda(F_\lambda M)=\bigoplus_{g\in G}{}^gM$.
\end{rmk}

The following lemma shows that $F_\lambda$ is stable under the $G$-action. 
\begin{lemma}\label{stableG}
For all $g\in G$, $F_\lambda {}^gM=F_\lambda M$ and $\hat{G}_{F_\lambda M}=\hat{G}$;
\end{lemma}
\begin{proof}
By definition, $F_\lambda {}^gM= \bigoplus_{i_0\in \tilde I} F_\lambda {}^gM(\bar {e_{i_0}})$ where, for each $\bar {e_{i_0}}\in \Lambda G$, we have $$F_\lambda {}^gM(\bar {e_{i_0}})= \bigoplus_{F(x)=\bar {e_{i_0}}}{}^gM(x)= \bigoplus_{F(x)=\bar {e_{i_0}}}M({}^gx)= \bigoplus_{F(x)=\bar {e_{i_0}}}M(x)=F_\lambda M(\bar {e_{i_0}}).$$ Here, the third equality holds as $F$ is stable under the $G$-action. Moreover, it is clear that $F_\lambda {}^gM(\alpha)=F_\lambda M(\alpha)$ for each $\alpha\in \Lambda_1G(\bar {e_{i_0}},\bar {e_{j_0}})$. This completes the proof.

Moreover, ${}^gF_\lambda M=F_\lambda {}^gM=F_\lambda M$ which says that $\hat{G}_{F_\lambda M}=\hat{G}$.
\end{proof}

The next proposition shows that, unlike Galois covering, the Galois semi-covering functor does not necessarily send an indecomposable module to an indecomposable module. 

\begin{proposition}\label{rem} 
If $M\in \mathrm{ind}\mbox{-}\Lambda$ with $G_{M}=G$ then $F_\lambda M=\bigoplus_{\hat{g}\in \hat{G}} \hat{g} \bar M$ for some $\bar M \in \mathrm{ind}\mbox{-}\Lambda G$.
\end{proposition}
\begin{proof}
Let $M:=(M(i)_{i\in Q_0}, M(\alpha_{ij})_{\alpha_{ij}:i\to j\in Q_1})$ be a module over $\Lambda$ that satisfies $G_M=G$. The support of a module $M$, denoted $\mathrm{Sup}M$, is defined as: $\mathrm{Sup}(M)=\{i\in Q_0: M(i)\neq 0\}$. 

First, we partition both supports. Write $\mathrm{Sup}(M)= S'\coprod S''$, where, $$S':=\{i\in \mathrm{Sup}(M)\mid G_i=G\}.$$ $$S'':=\{j\in \mathrm{Sup}(M)\mid G_j\neq G\}.$$ Moreover, write $S'=\coprod_{i=1}^2 S'_i$ and $S''=\coprod_{i=1}^2 S''_i$, $$S'_1:=\{i\in S' \mid M(\alpha_{ij})\cup M(\alpha_{ji})=0 \mbox{ for all } j\in S''\}.$$ $$S'_2:=\{i\in S' \mid M(\alpha_{ij})\cup M(\alpha_{ji})\neq 0 \mbox{ for some } j\in S''\}.$$ $$S''_1:=\{j\in S'' \mid M(\alpha_{ij})\cup M(\alpha_{ji})=0 \mbox{ for all } i\in S'\}.$$ $$S''_2:=\{j\in S'' \mid M(\alpha_{ij})\cup M(\alpha_{ji})\neq 0 \mbox{ for some } i\in S'\}.$$ 

Split the set of morphisms of $M$ as $\{M(\alpha_{ij})\mid {\alpha_{ij} \in Q_1}\}:=A'_1\coprod A'_2\coprod A''_1\coprod A''_2$ where,
$$A'_1:=\{M(\alpha_{ij}) \mid  i,j\in S'_1\}, A'_2:=\{M(\alpha_{ij}) \mid i\in S'_2, j\in S''_2\}.$$ $$A''_1:=\{M(\alpha_{ij}) \mid i, j\in S''_1\}, A''_2:=\{M(\alpha_{ij}) \mid i\in S''_2, j\in S'_2\}.$$ 

If $M(\alpha_{ij})\in A'_1$ then $F_\lambda (M(\alpha_{ij}))=\sum_{\rho \in \hat{G}} F_\lambda M(\bar{\alpha_{ij}}^\rho)$ by the third clause in Definition \ref{dgsc} where, $F_\lambda M(\bar{\alpha_{ij}}^\rho)\in {}_{\Lambda G}(F_\lambda M({e_{i_0,\rho}}),F_\lambda M({e_{j_0,\rho}}))$.

If $M(\alpha_{ij})\in A'_2$ then $F_\lambda (M(\alpha_{ij}))=\sum_{\rho \in \hat{G}} F_\lambda M(\bar{\alpha_{ij}}^\rho)$ by the second clause in Definition \ref{dgsc} where, $F_\lambda M(\bar{\alpha_{ij}}^\rho)\in {}_{\Lambda G}(F_\lambda M({e_{i_0,\rho}}),F_\lambda M({e_{j_0,tr}}))$. 

If $M(\alpha_{ij})\in A''_1$ then $F_\lambda (M(\alpha_{ij}))= \bar{\alpha_{ij}}$ where, $\bar{\alpha_{ij}}\in {}_{\Lambda G}(F_\lambda M({e_{i_0,tr}}),F_\lambda M({e_{j_0,tr}}))$.

If $M(\alpha_{ij})\in A''_2$ then $F_\lambda (M(\alpha_{ij}))=\sum_{\rho \in \hat{G}} F_\lambda M(\bar{\alpha_{ij}}^\rho)$ by the first clause in Definition \ref{dgsc} where, $F_\lambda M(\bar{\alpha_{ij}}^\rho)\in {}_{\Lambda G}(F_\lambda M({e_{i_0,tr}}), F_\lambda M({e_{j_0,\rho}}))$. 

If $G_{M}=G$ then $\mathcal{O}(i)\subseteq \mathrm{Sup}(M)$ for each $i\in \mathrm{Sup}(M)$. We compute the representation $F_\lambda M=(F_\lambda(M(i))_{i\in Q_0}, F_\lambda(M(\alpha_{ij}))_{\alpha_{ij}:i\to j\in Q_1})$, where 
\begin{equation*}
\begin{split}
\bigoplus_{i,j\in Q_0} F_\lambda (M(\alpha_{ij})) & = \bigoplus_{i,j\in S'_1} F_\lambda (M(\alpha_{ij}))\bigoplus_{i\in S'_2,j\in S''_2} F_\lambda (M(\alpha_{ij}))\bigoplus_{i,j\in S''_1} F_\lambda (M(\alpha_{ij}))\bigoplus_{i\in S''_2, j\in S'_2} F_\lambda (M(\alpha_{ij})) \\
& =\bigoplus_{i,j\in S'_1} \bigoplus_{\rho \in \hat{G}} F_\lambda M(\bar{\alpha_{i,j}}^\rho)\bigoplus_{i\in S'_2,j\in S''_2}\bigoplus_{\rho \in \hat{G}} F_\lambda M(\bar{\alpha_{ij}}^\rho)\bigoplus_{i,j\in S''_1} \bar{\alpha_{ij}} \bigoplus_{i\in S''_2, j\in S'_2} \bigoplus_{\rho \in \hat{G}} F_\lambda M(\bar{\alpha_{ij}}^\rho).
\end{split}
\end{equation*}

This ensures that $F_\lambda M= \bigoplus_{\rho \in \hat{G}} {}^\rho \bar{M}$ where, ${}^\rho \bar{M}:=({}^\rho \bar{M}(\tilde{i}))_{\tilde{i}\in Q_{G_0}}, {}^\rho \bar{M}(\tilde{\alpha}_{ij})_{\tilde{\alpha}_{ij}:\tilde{i}\to \tilde{j}\in Q_{G_1}})$ with

$\bigoplus_{i,j\in Q_0} {}^\rho \bar{M}(\tilde{\alpha}_{ij}):=\bigoplus_{i,j\in S'_1} F_\lambda M(\bar{\alpha_{i,j}}^\rho)\bigoplus_{i\in S'_2,j\in S''_2} F_\lambda M(\bar{\alpha_{ij}}^\rho)\bigoplus_{i,j\in S''_1} \bar{\alpha_{ij}} \bigoplus_{i\in S''_2, j\in S'_2} F_\lambda M(\bar{\alpha_{ij}}^\rho)$. 
\end{proof}

The next proposition shows that the Galois semi-covering functor $F_\lambda$ preserves indecomposability under some additional hypothesis.
    
\begin{proposition}\label{Ind}
Suppose $M\in \mathrm{ind}\mbox{-}\Lambda$ with $G_{M}\neq G$. Then $F_\lambda M$ is indecomposable.
\end{proposition} 
\begin{proof}
If possible, $F_\lambda M=N_1\bigoplus N_2$. Then by Remark \ref{adja}, we have $\bigoplus_{g\in G}{}^gM=G_\lambda(F_\lambda M)= G_\lambda(N_1)\bigoplus G_\lambda(N_2)$ which implies that $G_\lambda(N_1)=\bigoplus_{h\in H}{}^hM$ for some $H\subseteq G$. Moreover, we have $G_{F_\lambda M}=G$ by Lemma \ref{stableG} and thus, $G_\lambda(N_1)=\bigoplus_{h\in H}{}^{gh}M$ for all $g\in G$. Therefore, $G_\lambda(N_1)=\bigoplus_{g\in G}{}^gM$ and hence $N_2=0$. Hence the proof.
\end{proof}

\begin{rmk}\label{indr}
Lemma \ref{stableG} and Proposition \ref{Ind} ensures that for a $M\in \mathrm{ind}\mbox{-}\Lambda$ with $G_{M}\neq G$ we have $F_\lambda M=\bar M$ where, $\bar M\in \mathrm{ind}\mbox{-}\Lambda G$ with $\hat{G}_{\bar M}=\hat{G}$. Moreover, one can follow the algorithm described in Proposition \ref{rem} for the explicit construction of $\bar M$. 
\end{rmk}

Proposition \ref{rem} establishes a link of modules in $\mathrm{mod}\mbox{-}\Lambda$ and $\mathrm{mod}\mbox{-}\Lambda G$, whereas the following proposition shows the correspondence between their morphisms.

\begin{proposition}\label{3.7}
Assume that $f_{MN} \in \mathrm{Hom}_\Lambda(M,N)$ for some $M,N \in \mathrm{mod}\mbox{-}\Lambda$. Then there exists a $\bar{f}$ with $\bar{f}_{{}^{\hat{g}}\bar{M}{}^{\hat{g}}\bar{N}} \in \mathrm{Hom}_{\Lambda G}({}^{\hat{g}}\bar{M},{}^{\hat{g}}\bar{N})$ such that the following hold: 

$$F_\lambda(f_{{}^gM,{}^gN})\approx \begin{cases}\uwithtext{$\bar{N}$}_{\hat{g}\in \hat{G}} \bar{f}_{{}^{\hat{g}}\bar{M} \bar{N}} &\mbox{ if } G_{M}= G, G_{N}\neq G;\\\uwithtext{$\bar{M}$}_{\hat{g}\in \hat{G}} \bar{f}_{\bar{M} {}^{\hat{g}}\bar{N}} &\mbox{ if }G_{M}\neq G, G_{N}= G;\\\bigoplus_{\hat{g}\in \hat{G}} \bar f_{{}^{\hat{g}}\bar{M}{}^{\hat{g}}\bar{N}} &\mbox{ if } G_{MN}=G;\\\bar{f}_{\bar{M}\bar{N}}&\mbox{ if }G_{M}, G_{N}\neq G.\end{cases}$$

Where, $\uwithtext{Z}_{i\in I}f_i$ means $f_i$'s for an index set $I$ and $i\in I$ being glued at the module $Z$.
\end{proposition}

The next example is useful throughout the paper, as we refer to it later, also to explain the irreducible morphisms and almost split sequences in the skew group algebra as well.
\begin{example}\label{ARQSKW}
Consider the algebra $\Lambda$ and its skew algebra $\bar{\Lambda}=\Lambda \mathbb{Z}_2$ under the action of $\mathbb{Z}_2=\{e,g\}$ where, $g$ exchanges $3$ and $4$ and fixes the remaining vertices, given by the quiver with relations in Figure \ref{fig 5} and Figure \ref{fig 6} respectively. Note that, $F(e_1)=\bar {e_3}$, $F(e_2)=\bar {e_2}$, $F(e_3)=\bar {e_1}=e_1$ and $F(e_4)=\bar {e_1}= e_1$. Their AR-quivers are those of Figure \ref{fig 7} and Figure \ref{fig 8}, respectively, where the modules are represented by their composition factors. 

\begin{figure}[h]
\begin{minipage}[b]{0.41\linewidth}
\centering
\xymatrix@R=0.5mm{&&3\ar_\gamma[ld]\\
                                        1 \ar^\alpha @/^/[r] & \ar^\beta @/^/[l] 2& \\
                                        &&4\ar^\delta[lu]}
\caption{$\Lambda$ with $\rho= \{\alpha\beta\alpha, \beta\alpha\beta, \gamma\beta, \delta\beta \}$}
\label{fig 5}
\end{minipage}
\hspace{.35cm}
\begin{minipage}[b]{0.55\linewidth}
\centering
\xymatrix@R=0.5mm{ & 2\ar^\beta@/^/[r] & \ar^\gamma@/^/[l]  3\\
1 \ar^\alpha[ru] \ar_\delta[rd] & & \\
 & 4\ar^\epsilon@/^/[r] & \ar^\mu@/^/[l] 5}
\caption{$\bar{\Lambda}$ with $\rho=\{\alpha\beta, \beta\gamma\beta,  \gamma\beta\gamma, \delta\epsilon, \epsilon\mu\epsilon, \mu\epsilon\mu\}$}
    \label{fig 6}
    \end{minipage}
\end{figure}

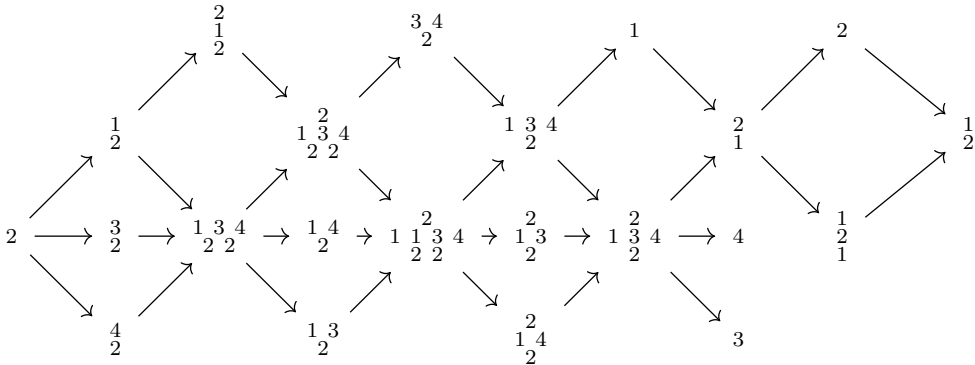
\begin{figure}
    \centering
\begin{tikzcd}[sep={3.55em,between origins}]
& &  \nd{2\\1\\2} \arrow[rd]
& & \nd{3\ 4 \\2} \arrow[rd]
& & \nd{1} \arrow[rd]
& & \nd{2} \arrow[rd] \\
%
& \nd{1\\2} \arrow[ru] \arrow[rd]
& &  \nd{2\\1\ 3\ 4\\2\ 2} \arrow[ru] \arrow[rd]
& & \nd{1\ 3\ 4\\2} \arrow[ru] \arrow[rd]
& & \nd{2\\1} \arrow[ru] \arrow[rd]
& & \nd{1\\2} \\
%
\nd{2} \arrow[ru] \arrow[rd] \arrow[r]
& \nd{3 \\2 } \arrow[r]
& \nd{1\ 3\ 4\\2\ 2} \arrow[rd] \arrow[ru] \arrow[r]
& \nd{1\ 4\\2} \arrow[r]
& \nd{2\\1\ 1\ 3\ 4\\2\ 2} \arrow[ru] \arrow[rd] \arrow[r]
& \nd{2\\1\ 3\\2} \arrow[r]
& \nd{2\\1\ 3\ 4\\2} \arrow[ru] \arrow[rd] \arrow[r]
& \nd{4} 
& \nd{1\\2\\1} \arrow[ru]
& & \\
%
& \nd{4\\2} \arrow[ru]
& & \nd{1\ 3\\2} \arrow[ru] 
& & \nd{2\\1\ 4\\2} \arrow[ru]
& & \nd{3} 
\end{tikzcd}
\caption{AR quiver of $\Lambda$}
    \label{fig 7}
\end{figure}

\begin{figure}
\centering
\begin{tikzcd}[sep={3.0em,between origins}]
& &  \nd{2\\3\\2} \arrow[rd]
& & \nd{1\\4} \arrow[rd]
& & \nd{5} \arrow[rd]
& & \nd{4} \arrow[rd] \\
%
& \nd{3\\2} \arrow[ru] \arrow[rd]
& &  \nd{2\\3\ 1\\2\ 4} \arrow[ru] \arrow[rd]
& & \nd{1\ 5\\4} \arrow[ru] \arrow[rd]
& & \nd{4\\5} \arrow[ru] \arrow[r]
& \nd{5\\4\\5} \arrow[r]
& \nd{5\\4} \\
%
\nd{2} \arrow[ru] \arrow[rd]
& & \nd{3 1\\2 4} \arrow[rd] \arrow[ru]
& & \nd{2\\3\ 1\ 5\\2\ 4} \arrow[ru] \arrow[rd]
& & \nd{1\  4\\5\\4} \arrow[ru] \arrow[rd]
& & \\
%
& \nd{1\\24} \arrow[ru] \arrow[rd]
& & \nd{315\\24} \arrow[ru] \arrow[rd]
& & \nd{2\ 4\\3\ 1\ 5\\2\ 4} \arrow[ru] \arrow[rd]
& & \nd{1} \\
%
\nd{4} \arrow[ru] \arrow[rd]
& & \nd{15\\24} \arrow[ru] \arrow[rd]
& & \nd{4\\3\ 1\ 5\\2\ 4} \arrow[ru] \arrow[rd]
& & \nd{2\\3\ 1\\2} \arrow[ru] \arrow[rd] \\
%
& \nd{5\\4} \arrow[ru] \arrow[rd]
& & \nd{4\\1\ 5\\2\ 4} \arrow[ru] \arrow[rd]
& & \nd{1\\3\ 2} \arrow[ru] \arrow[rd]
& & \nd{2\\3} \arrow[rd] \arrow[r]
& \nd{3\\2\\3} \arrow[r]
& \nd{3\\2} \\
%
& & \nd{4\\5\\4} \arrow[ru]
& & \nd{1\\2} \arrow[ru]
& & \nd{3} \arrow[ru]
& & \nd{2} \arrow[ru]
\end{tikzcd}
\caption{AR quiver of $\Lambda \mathbb{Z}_2$}
    \label{fig 8}
\end{figure}
Let $M = 2$ be the simple $\Lambda$-module associated to vertex 2. Observe this module is stable by the action of $G=\mathbb{Z}_2$ over $\Lambda$. Here, $F_\lambda(2)= 2 \oplus 4$ is the direct sum of two simple modules over $\Lambda \mathbb{Z}_2$, i.e. $F_\lambda(M)$ is a decomposable module. 
Let $N= \nd{3\\2}$ be the $\Lambda$-module. This module is not stable since ${}^gN=\nd{4\\2}\neq N$. Note that $F_\lambda(N)=F_\lambda({}^gN)=\bar{N}= \nd{1\\2\ 4}$ which is indecomposable. We will prove this is always the case for unstable modules. 

Now, consider the irreducible morphisms $f:2 \to \begin{matrix} 3\\ 2 \end{matrix}$ and ${}^gf:2 \to \begin{matrix} 4\\ 2 \end{matrix}$. The image of $f$ is the irreducible morphism $F_\lambda(f)=\uwithtext{$\bar{N}$}_{\hat{g}\in \mathbb{Z}_2}{}^{\hat{g}}\bar{f}:2 \oplus 4 \to \nd{1\\2\ 4}$ where, $\bar{f}:2 \to \nd{1\\2\ 4}$ and ${}^{\hat{g}}\bar{f}: 4 \to \nd{1\\2\ 4}$. On the other hand, both $2$ and $\begin{matrix} 1\\ 2 \end{matrix}$ are stable under the $G$-action and $F_\lambda$ takes the irreducible morphism $h:2 \to \begin{matrix} 1\\ 2 \end{matrix}$ to two different irreducible morphisms  $\bar h:2 \to \begin{matrix} 3\\ 2 \end{matrix}$ and ${}^{\hat{g}}\bar h: 4 \to \begin{matrix} 5\\ 4 \end{matrix}$.

Moreover, one can easily check that, the module $\begin{matrix} 3\\ 2 \end{matrix}$ and the morphism $2 \to \begin{matrix} 3\\ 2 \end{matrix}$ have no preimage under $F_\lambda$. This means that $F_\lambda$ is not dense. 

\end{example}

The next result establishes a Galois semi-covering functor from $\mathrm{mod}\mbox{-}\Lambda$ to $\mathrm{mod}\mbox{-}\Lambda G$.

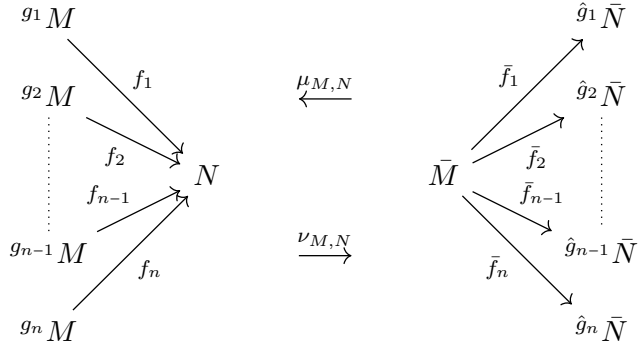
\begin{theorem}\label{diag}
Assume that $G$ acts on an algebra $\Lambda$ and $\Lambda G$ is the associated skew group algebra. Then for any $M, N\in \mathrm{mod}\mbox{-}\Lambda$, the functor $F_\lambda: \mathrm{mod}\mbox{-}\Lambda \to \mathrm{mod}\mbox{-}\Lambda G$ induces the following isomorphisms of vector spaces: 
\vspace{-0.1in} 
$$\mathrm{Hom}_{\Lambda G}(F_\lambda M,F_\lambda N)\approx \begin{cases}\bigoplus_{g\in G} \mathrm{Hom}_{\Lambda}({}^gM,N)&\mbox{ if } G_{M}\neq G;\\\bigoplus_{g\in G} \mathrm{Hom}_{\Lambda}(M,{}^gN)&\mbox{ if }G_{N}\neq G;\\\mathrm{Hom}_{\Lambda}^{|G|}(M,N)&\mbox{ if }G_{MN}= G.\end{cases}$$
In particular, $F_\lambda$ is a Galois semi-covering functor from $\mathrm{mod}\mbox{-}\Lambda$ to $\mathrm{mod}\mbox{-}\Lambda G$.
\end{theorem}

\begin{proof}
Here, we analyze the first two cases, taking into account three different situations where $G_{M}\neq G, G_{N}= G$; $G_{N}\neq G, G_{M}= G$ and $G_{N}\neq G, G_{M}\neq G$. Now, let us describe the isomorphism $\nu_{M,N}$ explicitly as follows:

\noindent{\textbf{Case-I:}} Here, $G_{M}\neq G, G_{N}= G$. Thus we have $F_\lambda ({}^gM)=\bar M$ for all $g\in G$ and $\bar M \in \mathrm{mod}\mbox{-}\Lambda G$ and $F_\lambda (N)=\bigoplus_{\hat{g}\in \hat{G}} {}^{\hat{g}} \bar N$ for some $\bar N \in \mathrm{mod}\mbox{-}\Lambda G$ by Proposition \ref{rem} and Remark \ref{indr}. We define $\nu_{M, N}: \bigoplus_{g\in G} \mathrm{Hom}_{\Lambda}({}^gM,N) \to \mathrm{Hom}_{\Lambda G}(F_\lambda M,F_\lambda N)$ by $\nu_{M,N}(f_1,\hdots, f_n)_{1\times n}=(\bar{f_1},\hdots, \bar{f_n})_{n\times 1}$, and the inverse isomorphism $\mu_{M, N}$ by $\mu_{M, N}(\bar{f_1},\hdots, \bar{f_n})_{n\times 1}=(f_1,\hdots, f_n)_{1\times n}$, where, $f_i=f_{{}^{g_i}MN}: {}^{g_i} M\to N$ and $\bar{f_i}=\bar{f}_{\bar{M}{}^{\hat{g}_i}\bar{N}}: \bar{M}\to {}^{\hat{g}_i}\bar{N}$ are morphisms in $\mathrm{mod}\mbox{-}\Lambda$ and $\mathrm{mod}\mbox{-}\Lambda G$ respectively. Clearly, $\nu_{M,N}$ and $\mu_{M,N}$ are linear morphisms.

\begin{figure}
\centering
\begin{tikzcd} [sep={2.7em,between origins}]
{}^{g_1}M \arrow[rrdd, "f_1"]                             &  &   &                             &                              &                                                                                                                     &  & {}^{\hat{g}_1} \bar{N}                             \\
{}^{g_2}M \arrow[rrd, "f_2"'] \arrow[dd, no head, dotted] &  &   & {}                          & {} \arrow[l, "{\mu_{M,N}}"'] &                                                                                                                     &  & {}^{\hat{g}_2} \bar{N} \arrow[dd, no head, dotted] \\
                                                          &  & N &                             &                              & \bar{M} \arrow[rruu, "\bar{f}_1"] \arrow[rru, "\bar{f}_2"'] \arrow[rrdd, "\bar{f}_n"'] \arrow[rrd, "\bar{f}_{n-1}"] &  &                                                    \\
{}^{g_{n-1}}M \arrow[rru, "f_{n-1}"]                      &  &   & {} \arrow[r, "{\nu_{M,N}}"] & {}                           &                                                                                                                     &  & {}^{\hat{g}_{n-1}} \bar{N}                         \\
{}^{g_n}M \arrow[rruu, "f_n"']                            &  &   &                             &                              &                                                                                                                     &  & {}^{\hat{g}_n} \bar{N}                            
\end{tikzcd}
\end{figure}

\noindent{\textbf{Case-II:}} Here, $G_{N}\neq G, G_{M}= G$. Thus we have $F_\lambda ({}^gN)=\bar N$ for all $g\in G$ and $\bar N \in \mathrm{mod}\mbox{-}\Lambda G$ and $F_\lambda (M)=\bigoplus_{\hat{g}\in \hat{G}} {}^{\hat{g}} \bar M$ for some $\bar M \in \mathrm{mod}\mbox{-}\Lambda G$ by Proposition \ref{rem} and Remark \ref{indr}. We define $\nu_{M,N}:\bigoplus_{g\in G} \mathrm{Hom}_{\Lambda}(M,{}^gN)\to \mathrm{Hom}_{\Lambda G}(F_\lambda M,F_\lambda N)$ by $\nu_{M,N}(f_1,\hdots, f_n)_{n\times 1}=(\bar{f_1},\hdots, \bar{f_n})_{1\times n}$ and the inverse isomorphism $\mu_{M, N}$ by $\mu_{M, N}(\bar{f_1},\hdots, \bar{f_n})_{1\times n}=(f_1,\hdots, f_n)_{n\times 1}$, where, $f_i=f_{M{}^{g_i}N}: M\to {}^{g_i} N$ and $\bar{f_i}=\bar{f}_{{}^{\hat{g}_i}\bar{M}\bar{N}}: {}^{\hat{g}_i}\bar{M}\to \bar{N}$ are morphisms in $\mathrm{mod}\mbox{-}\Lambda$ and $\mathrm{mod}\mbox{-}\Lambda G$ respectively. Clearly, $\nu_{M,N}$ and $\mu_{M,N}$ are linear morphisms.

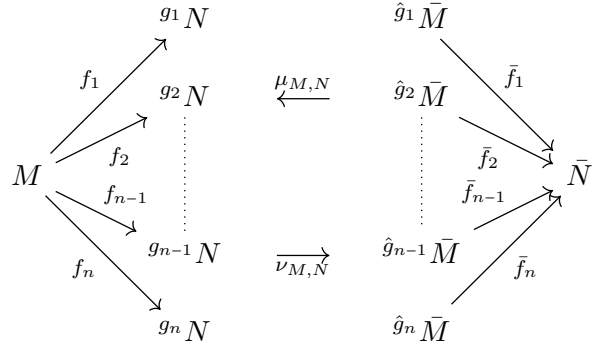
\[\begin{tikzcd}[sep={2.7em,between origins}]
                                                                                      &  & {}^{g_1}N                             &                              &                              & {}^{\hat{g}_1}\bar{M} \arrow[rrdd, "\bar{f}_1"]                             &  &         \\
                                                                                      &  & {}^{g_2}N \arrow[dd, no head, dotted] & {}                           & {} \arrow[l, "{\mu_{M,N}}"'] & {}^{\hat{g}_2}\bar{M} \arrow[rrd, "\bar{f}_2"'] \arrow[dd, no head, dotted] &  &         \\
M \arrow[rruu, "f_1"] \arrow[rru, "f_2"'] \arrow[rrdd, "f_n"'] \arrow[rrd, "f_{n-1}"] &  &                                       &                              &                              &                                                                             &  & \bar{N} \\
                                                                                      &  & {}^{g_{n-1}}N                         & {} \arrow[r, "{\nu_{M,N}}"'] & {}                           & {}^{\hat{g}_{n-1}}\bar{M} \arrow[rru, "\bar{f}_{n-1}"]                      &  &         \\
                                                                                      &  & {}^{g_n}N                             &                              &                              & {}^{\hat{g}_n}\bar{M} \arrow[rruu, "\bar{f}_n"']                            &  &        
\end{tikzcd}\]

\noindent{\textbf{Case-III:}} Here, $G_M\neq G, G_N\neq G$. Then it works as a locally Galois covering as $G$-action does not fix $M$ and $N$ both, and as a result, the number of arrows from ${}^gM$ to $N$ for $g\in G$ is equal to the set of arrows from $\bar{M}$ to $\bar{N}$ (= the number of arrows from $M$ to ${}^gN$ also) i.e. $$\mathrm{dim}_K\mathrm{Hom}_{\Lambda G}(\bar{M}, \bar{N})= \mathrm{dim}_{K}\mathrm{Hom}_{\Lambda}(\bigoplus_{g\in G}{}^gM, N)=\mathrm{dim}_{K}\bigoplus_{g\in G}\mathrm{Hom}_{\Lambda}({}^gM, N).$$

\[
\begin{tikzcd}[sep={2.85em,between origins}]
{}^{g_1}M \arrow[rrdd, "f_1"]                             &  &                              & {} \arrow[rr, "\bar{f}_1"]                           &                                & {}                          &                                                                                       &  & {}^{g_1}N                             \\
{}^{g_2}M \arrow[rrd, "f_2"'] \arrow[dd, no head, dotted] &  & {}                           & {} \arrow[l, "{\mu_{M,N}}"'] \arrow[rr, "\bar{f}_2"] & {} \arrow[dd, no head, dotted] & {} \arrow[r, "{\mu_{M,N}}"] & {}                                                                                    &  & {}^{g_2}N \arrow[dd, no head, dotted] \\
                                                          &  & N                            & \bar{M}                                              &                                & \bar{N}                     & M \arrow[rruu, "f^1"] \arrow[rru, "f^2"'] \arrow[rrd, "f^{n-1}"] \arrow[rrdd, "f^n"'] &  &                                       \\
{}^{g_{n-1}}M \arrow[rru, "f_{n-1}"]                      &  & {} \arrow[r, "{\nu_{M,N}}"'] & {} \arrow[rr, "\bar{f}_{n-1}"]                       & {}                             & {}                          & {} \arrow[l, "{\nu_{M,N}}"]                                                           &  & {}^{g_{n-1}}N                         \\
{}^{g_n}M \arrow[rruu, "f_n"']                            &  &                              & {} \arrow[rr, "\bar{f}_n"]                           &                                & {}                          &                                                                                       &  & {}^{g_n}N                            
\end{tikzcd}
\]

The second equality holds since it is a finite-dimensional vector space. Thus we get $$\mathrm{Hom}_{\Lambda G}(\bar{M}, \bar{N})\approx\bigoplus_{g\in G}\mathrm{Hom}_{\Lambda}({}^gM, N)\approx\bigoplus_{g\in G}\mathrm{Hom}_{\Lambda}(M, {}^gN).$$ 

Let $m\leq n$ be the number of the non-zero representatives $f_i$ from the equivalence classes $\mathcal{O}_{f_i}$ under the $G$-action. Then the isomorphism $\nu_{M, N}: \bigoplus_{g\in G} \mathrm{Hom}_{\Lambda}({}^gM,N) \to \mathrm{Hom}_{\Lambda G}(\bar{M}, \bar{N})$ is given by $\nu_{M,N}(f_1,\hdots, f_m)_{1\times m}=(\bar{f_1},\hdots, \bar{f_m})$ and the inverse $\mu_{M, N}$ by $\mu_{M, N}(\bar{f_1},\hdots, \bar{f_m})=(f_1,\hdots, f_m)_{1\times m}$ where, $f_i=f_{{}^{g_i}MN}: {}^{g_i} M\to N$ and $\bar{f_i}=\bar{f}_{\bar{M}\bar{N}}: \bar{M}\to \bar{N}$ are morphisms in $\mathrm{mod}\mbox{-}\Lambda$ and $\mathrm{mod}\mbox{-}\Lambda G$ respectively. Some of the $f_i$'s could be zero morphisms in the above picture (see Example \ref{cycthr}). Moreover, the second isomorphism in the above expression is defined similarly.

Note that, we have the following equality: $$\mathrm{dim}_K\mathrm{Hom}^{|\tilde{G}|}_{\Lambda G}(\bar{M}, \bar{N})= \mathrm{dim}_K\mathrm{Hom}_{\Lambda}(\bigoplus_{g\in G}{}^gM, \bigoplus_{g\in G}{}^gN).$$

\noindent{\textbf{Case-IV:}} Here, $G_{MN}= G$. Hence, $F_\lambda (M)=\bigoplus_{\hat{g}\in \hat{G}} {}^{\hat{g}} \bar M$ and $F_\lambda (N)=\bigoplus_{\hat{g}\in \hat{G}} {}^{\hat{g}} \bar N$ for some $\bar M, \bar N \in \mathrm{mod}\mbox{-}\Lambda G$ by Proposition \ref{rem}. From the above discussion, it is clear that, $$\mathrm{dim}_K\mathrm{Hom}^{|G|}_{\Lambda}(M, N)= \mathrm{dim}_K\mathrm{Hom}_{\Lambda G}(\bigoplus_{\hat{g}\in \hat{G}}{}^{\hat{g}}\bar{M}, \bigoplus_{\hat{g}\in \hat{G}}{}^{\hat{g}}\bar{N})= \mathrm{dim}_K\mathrm{Hom}_{\Lambda G}(F_\lambda (M), F_\lambda (N)).$$ Therefore, we have, $\mathrm{Hom}_{\Lambda G}(F_\lambda (M), F_\lambda (N))=\mathrm{Hom}^{|G|}_{\Lambda}(M, N)$, since it is a finite-dimensional vector space. Here, the explicit isomorphism $\nu_{M,N}: \mathrm{Hom}^{|G|}_{\Lambda}(M, N)\to \mathrm{Hom}_{\Lambda G}(F_\lambda (M), F_\lambda (N))$ is given by sending $f_i\mapsto \bigoplus_{g\in G} {}^g\bar{f_i}$. One can easily verify that it is a monomorphism; hence, the isomorphism follows, since both vector spaces have the same dimension.
\end{proof}

Let us discuss an example to understand the last two cases in the above theorem.

\begin{example}\label{cycthr}
Consider the skew group algebra $\bar{\Lambda}:=\Lambda\mathbb{Z}_3$ of the Kronecker algebra $\Lambda$ where the action of $\mathbb{Z}_3:=\{e,g_1, g_2\}$ on $\Lambda$ is given by fixing all the arrows, i.e. $g_i(\alpha)=\alpha$ and $g_i(\beta)=\beta$ for all $i=1,2$. Let $\mathcal{M}$ be a representation of $\Lambda$ given below.

\begin{figure}[h]
\begin{minipage}[b]{0.45\linewidth}
\centering
\begin{tikzcd}[sep={2.8em,between origins}]
1 \arrow[r, "\beta"', bend right=49] \arrow[r, "\alpha", bend left=49] & 2
\end{tikzcd}

\caption{$\Lambda$}
\label{fig:my_label}
\end{minipage}
\hspace{0.45cm}
\begin{minipage}[b]{0.5\linewidth}
\centering
\begin{tikzcd}[sep={2.8em,between origins}]
M_1 \arrow[r, "f_2"', bend right=49] \arrow[r, "f_1", bend left=49] & M_2
\end{tikzcd}
\caption{$\mathcal{M}$}
    \label{fig:my_label}
    \end{minipage}
\end{figure}

\begin{figure}[h]
\begin{minipage}[b]{0.45\linewidth}
\centering
\begin{tikzcd}[sep={2.9em,between origins}]
1 \arrow[rr, "\alpha"] \arrow[rrd, "\beta"]                      &  & 2   \\
1' \arrow[rr, "{}^{g_1}\alpha"] \arrow[rrd, "{}^{g_1}\beta"', shift right] &  & 2'  \\
1'' \arrow[rr, "{}^{g_2}\alpha"'] \arrow[rruu, "{}^{g_2}\beta"']           &  & 2''
\end{tikzcd}    
\caption{$\Lambda G$}
\label{fig:my_label}
\end{minipage}
\hspace{0.45cm}
\begin{minipage}[b]{0.5\linewidth}
\centering
\begin{tikzcd}[sep={2.9em,between origins}]
\bar{M} \arrow[rr, "\bar{f_1}"] \arrow[rrd, "\bar{f_2}"]                      &  & \bar{N}   \\
{}^{g_1}\bar{M} \arrow[rr, "{}^{g_1}\bar{f_1}"] \arrow[rrd, "{}^{g_1}\bar{f_2}"', shift right] &  & {}^{g_1}\bar{N}  \\ {}^{g_2}\bar{M} \arrow[rr, "{}^{g_2}\bar{f_1}"'] \arrow[rruu, "{}^{g_2}\bar{f_2}"']           &  & {}^{g_2}\bar{N}
\end{tikzcd}    
\caption{$F_\lambda(\mathcal{M})$}
    \label{fig:my_label}
    \end{minipage}
\end{figure}

Here, $G_{M_1M_2}=G$ and $\mathrm{Hom}_{\Lambda G}(F_\lambda (M), F_\lambda (N))=\mathrm{Hom}_{\Lambda G}(\bigoplus_{g\in G} {}^g\bar{M}, \bigoplus_{g\in G} {}^g\bar{N})$ is described by the following matrix $A$:

$$A:=\left(\begin{matrix} \bar{f_1}& \bar{f_2}& 0\\ 0& {}^{g_1}\bar{f_1}& {}^{g_1}\bar{f_2} \\ {}^{g_2}\bar{f_1}& 0& {}^{g_2}\bar{f_2}  \end{matrix}\right)$$

Here, $\mu_{M,N}: \mathrm{Hom}_{\Lambda G}(F_\lambda (M), F_\lambda (N)) \to  \mathrm{Hom}^{|G|}_{\Lambda}(M, N)$ is given by- $$\left(\begin{matrix} h_1& h_2& 0\\ 0& h_3& h_4 \\ h_5& 0& h_6  \end{matrix}\right)\mapsto (\bar{h_1}, \bar{h_2}).$$ Whereas, $\nu_{M,N}: \mathrm{Hom}^{|G|}_{\Lambda}(M, N)\to \mathrm{Hom}_{\Lambda G}(F_\lambda (M), F_\lambda (N))$ is given by- $$(f_1,f_2)\mapsto A.$$
\end{example}

\end{section}

In Example \ref{ARQSKW}, we conclude that $F_\lambda$ is not dense over $\Lambda$, but a similar property holds, which we call the semi-dense property.

\begin{definition}\label{SDF}[Semi-dense property of a functor]
A functor $F: A\to B$ between two linear categories $A$ and $B$ is semi-dense if for any $Y\in \mathrm{Ob}(B)$, there exists a $X\in \mathrm{Ob}(A)$ and $Z\in \mathrm{Ob}(B)$ such that $F(X)= Y\bigoplus Z$.      
\end{definition}

The following corollary demonstrates that $F_\lambda$ is semi-dense.
\begin{corollary}\label{SDPF}
Assume that $G$ acts on an algebra $\Lambda$ and $\Lambda G$ is the associated skew group algebra. Then the functor $F_\lambda: \mathrm{mod}\mbox{-}\Lambda \to \mathrm{mod}\mbox{-}\Lambda G$ is semi-dense.
\end{corollary}
\begin{proof}
For any $M\in \mathrm{mod}\mbox{-}\Lambda G$, choose $G_\lambda(M)\in \mathrm{mod}\mbox{-}\Lambda$ as we have $F_\lambda(G_\lambda(M))=\bigoplus_{g\in G}{}^{g}M$ by Remark \ref{adja}.
\end{proof}

\begin{section} {Stable rank of skew group algebras}

In this section, we show that the previously defined functor $F_\lambda$ is well-behaved to the powers of the radicals. We also show that stable rank is preserved under skew group algebra construction. As a result, we determine the stable ranks of skew gentle algebras.

\begin{definition}\label{ranstab}
The radical $\mathrm{rad}_\Lambda$ of $\Lambda$ is the ideal generated by the non-invertible morphisms between indecomposable objects. A morphism in $\Lambda$ is called radical if it lies in $\mathrm{rad}(\Lambda)$. Its powers are defined inductively as follows.
\begin{enumerate}
    \item $\mathrm{rad}^n_\Lambda= \mathrm{rad}^{n-1}_\Lambda\mathrm{rad}_\Lambda$ if $n$ is finite;
    \item $\mathrm{rad}^\alpha_\Lambda:= \bigcap_{\mu<\alpha} \mathrm{rad}^\mu_\Lambda$ if $\alpha$ is a limit ordinal;
    \item $\mathrm{rad}^\alpha_\Lambda:= (\mathrm{rad}^\mu_\Lambda)^{n+1}$ if $\alpha= \mu + n$ is a successor ordinal;
    \item $\mathrm{rad}^\infty_\Lambda:= \bigcap_\mu \mathrm{rad}^\mu_\Lambda$.
\end{enumerate}

There is a descending chain of ideals
$$\mathrm{mod}\mbox{-}\Lambda\supseteq\mathrm{rad}_\Lambda\supseteq\mathrm{rad}^2_\Lambda\supseteq\hdots\supseteq\mathrm{rad}^\omega_\Lambda\supseteq\mathrm{rad}^{\omega+1}_\Lambda\supseteq\hdots\supseteq\mathrm{rad}^\infty_\Lambda\supseteq0.$$
The \emph{rank} of $\Lambda$, $\mathrm{rank}(\Lambda)$, is the minimum $\alpha$, if exists, such that $\mathrm{rad}^\alpha_\Lambda=0$, otherwise the rank is $\infty$. The \emph{stable rank} of $\Lambda$ is the minimum $\alpha$ such that $\mathrm{rad}^\alpha_\Lambda=\mathrm{rad}^{\alpha+1}_\Lambda$. 
\end{definition}

Let $M=\bigoplus M_i$ and  $N=\bigoplus N_j$ be two decomposable modules where $M_i$ and $N_j$ are their indecomposable direct summands. Recall that a morphism $f:M \to N$ is in $\mathrm{rad}^n(M,N)$ if and only if $\beta_j f\alpha_{i}$ is in $\mathrm{rad}^n(M_i,N_j)$ for each $i$ and $j$ where $\alpha_{i}:M_i\to M$ are the inclusion maps and $\beta_j:N \to N_j$ are the projections. We verify this result for ordinal powers of the radical as well. As a direct consequence of proposition \ref{Ind}, we have the following result. 

\begin{lemma}\label{entradas}
Suppose $M$ is an indecomposable module and $L =\bigoplus^{m}_{i=1} L_i$ where each $L_i$ is an indecomposable module. If a morphism $f:L \to M$ is given by $f=(f_1,...,f_m)$ where $f_i:L_i \to M $, then for any ordinal $\alpha$, $f \in \mathrm{rad}^{\alpha}(L,M) \setminus \mathrm{rad}^{\alpha+1}(L,M)$  if and only if there exists a $f_j:L_j\to M$ such that $f_j \in \mathrm{rad}^{\alpha}(L_j,M) \setminus \mathrm{rad}^{\alpha+1}(L_j,M)$. 
\end{lemma}
\begin{proof}
Assume that for all $i$, $f_i \in \mathrm{rad}^{\alpha+1}(L_i,M)$ then we have that $f:L\to M$ is in $\mathrm{rad}^{\alpha+1}(L,M)$, a contradiction to our assumption. Hence there exists $f_j:L_j\to M$ such that $f_j \in \mathrm{rad}^{\alpha}(L_j,M) \setminus \mathrm{rad}^{\alpha+1}(L_j,M)$. 
Converse is similar, if there exists $f_j:L_j\to M$ such that $f_j \in \mathrm{rad}^{\alpha}(L_j,M) \setminus \mathrm{rad}^{\alpha+1}(L_j,M)$ then $f:L\to M$ is not in $\mathrm{rad}^{\alpha+1}(L,M).$
\end{proof}





We are now in a condition to state our main result for this section.

\begin{theorem}\label{rank}
Suppose an abelian group $G$ acts on an algebra $\Lambda$ and $F_{\lambda}: \mathrm{mod}\mbox{-}\Lambda \to \mathrm{mod}\mbox{-}\Lambda G$ is a Galois semi-covering. Then $F_{\lambda}$ preserves powers of radicals.
\end{theorem}
\begin{proof}
Let $M,N$ be indecomposable $\Lambda$-modules and $f:M \to N$ be a morphism.
 
We show that if $f \in \mathrm{rad}^{\alpha}(M,N) \setminus \mathrm{rad}^{\alpha +1}(M,N)$, then $F_\lambda(f)\in \mathrm{rad}^{\alpha}(F_\lambda(M),F_\lambda(N)) \setminus \mathrm{rad}^{\alpha+1} (F_\lambda(M),F_\lambda(N))$. We analyse this in consideration of the following cases.

\begin{enumerate}
    \item $G_{M}\neq G$ and $G_{N}\neq G$;
    \item $G_{M}=G$ and $G_{N}\neq G$;
    \item $G_{M}\neq G$ and $G_{N} = G$;
    \item $G_{M}=G$ and $G_{N} = G$.
\end{enumerate}

Assume that $f \in \mathrm{rad}^{\alpha}(M,N) \setminus \mathrm{rad}^{\alpha +1}(M,N)$. 

\noindent{\textbf{Case (1)}:} Since $G_{M}\neq G$ and $G_{N}\neq G$, both $F_{\lambda}(M)$ and $F_{\lambda}(N)$ are indecomposable by Proposition \ref{Ind}. Moreover, $\bar{f}=F_\lambda(f)$ can be identified with the morphism $(f,0,..,0) \in \bigoplus_{g\in G}\mathrm{Hom}_{\Lambda}(M,{}^gN)$ via the isomorphism given by Theorem \ref{diag}. Since $f \in \mathrm{rad}^{\alpha}(M,N) \setminus \mathrm{rad}^{\alpha +1}(M,N)$, we have $(f,0,..,0) \in \mathrm{rad}^{\alpha}(\oplus_{g\in G}{}^gM,N) \setminus \mathrm{rad}^{\alpha +1}(\oplus_{g\in G}{}^gM,N)$ by lemma \ref{entradas} and thus $F_\lambda(f)\in \mathrm{rad}^{\alpha} (F_\lambda(M),F_\lambda(N)) \setminus \mathrm{rad}^{\alpha+1}(F_\lambda(M),F_\lambda(N))$.

\noindent{\textbf{Case (2)}:} Since $G_{M}=G$ and $G_{N}\neq G$, we have $\mathrm{Hom}_{\Lambda G}(F_\lambda M,F_\lambda N)  \approx \bigoplus_{g\in G}\mathrm{Hom}_{\Lambda}(M,{}^gN)$ by theorem \ref{diag}. Hence, using this isomorphism $F_\lambda(f)$ can be identified with $({}^gf)_{g\in G}$ where ${}^gf:M \to {}^gN$ for all $g \in G$. Since $f \in \mathrm{rad}^{\alpha}(M,N) \setminus \mathrm{rad}^{\alpha+1}(M,N)$ then ${}^gf \in \mathrm{rad}^{\alpha}(M,{}^gN) \setminus \mathrm{rad}^{\alpha+1}(M,{}^gN)$ for each $g\in G$. Hence, by Lemma \ref{entradas}, $({}^gf)_{g\in G}\in \mathrm{rad}^{\alpha}(M, \oplus_{g \in G}{}^gN))\setminus \mathrm{rad}^{\alpha +1}(M, \oplus_{g \in G}{}^gN))$. Thus, $F_\lambda(f)\in \mathrm{rad}^{\alpha} (F_\lambda(M),F_\lambda(N)) \setminus \mathrm{rad}^{\alpha+1}(F_\lambda(M),F_\lambda(N))$.

\noindent{\textbf{Case (3)}:} The result follows dually to the previous case.

\noindent{\textbf{Case (4)}:} Since $G_{MN}=G$, we have $\Lambda^{|G|}(M,N) \approx\Lambda G(F_\lambda M,F_\lambda N)$ by Theorem \ref{diag}. Thus we identify $F_\lambda(f)$ with a diagonal matrix $\mathcal{M}$ with ${}^gf:{}^gM \to {}^gN = f:M \to N $ as the respective diagonal entry. Now, $f \in \mathrm{rad}^{\alpha}(M,N) \setminus \mathrm{rad}^{\alpha+1}(M,N)$ and hence the morphism given by $\mathcal{M}$ is in  $\mathrm{rad}^{\alpha}(\oplus_{g\in G} {}^gM,\oplus_{g\in G} {}^gN)$, since every entry is in the respective $\mathrm{rad}^{\alpha}$. Moreover, since $f \notin \mathrm{rad}^{\alpha+1}(M,N)$ neither does this morphism. Therefore, $F_\lambda(f)$ is a morphism in $\mathrm{rad}^{\alpha}(F_\lambda(M),F_\lambda(N)) \setminus \mathrm{rad}^{\alpha+1}(F_\lambda(M),F_\lambda(N))$ and our claim holds.
\end{proof}

A direct consequence of the result is that rank is also preserved under skewness.

\begin{theorem}
Suppose an abelian group $G$ acts on an algebra $\Lambda$ and $\Lambda G$ is the associated skew algebra. Then $\mathrm{rank} \Lambda = \mathrm{rank} \Lambda G$.
\end{theorem}

Below, we present a Galois semi-covering functor between $\mathrm{rad}\mbox{-}\Lambda$ and $\mathrm{rad}\mbox{-}\Lambda G$.

\begin{corollary}
Assume that $G$ acts on an algebra $\Lambda$ and $\Lambda G$ is the associated skew group algebra. Then for any ordinal $\alpha$ and $M, N\in \mathrm{mod}\mbox{-}\Lambda$, the functor $F_\lambda: \mathrm{mod}\mbox{-}\Lambda \to \mathrm{mod}\mbox{-}\Lambda G$ induces the following isomorphisms of vector spaces:
$$\mathrm{rad}^\alpha_{\Lambda G}(F_\lambda M,F_\lambda N)\approx \begin{cases}\bigoplus_{g\in G} \mathrm{rad}^\alpha_{\Lambda} (gM,N)&\mbox{ if } G_{M}\neq G;\\\bigoplus_{g\in G} \mathrm{rad}^\alpha_{\Lambda}(M,gN)&\mbox{ if }G_{N}\neq G;\\\mathrm{rad}_{\Lambda}^{\alpha|G|}(M,N)&\mbox{ if }G_{MN}= G.\end{cases}$$  
\end{corollary}

We end the section with a remark on the preservation of stable rank under skewness.
\begin{theorem}\label{stable}
Suppose an abelian group $G$ acts on an algebra $\Lambda$ and $\Lambda G$ is the associated skew group algebra. Then the stable rank of $\Lambda$ and $\Lambda G$ are the same.
\end{theorem}
\end{section}

\subsection{The stable rank of skew gentle algebras}
We give a brief description of a skew-gentle algebra and determine its stable rank. 
\begin{definition}
A gentle algebra $\Lambda$ is a bound quiver algebra $KQ/\langle \rho \rangle$, where $\rho$ is a set of monomial relations of length $2$ generating an ideal of the path algebra $KQ$ satisfying the following conditions:
\begin{enumerate} 
\item Any vertex of $Q_0$ has at most two indegrees and outdegrees;
\item For any arrow $b$, there is at most one arrow $c$ with $s(c)=t(b)$ and $bc\in \langle \rho \rangle$ and at most one arrow $a$ with $t(b)= s(a)$ and $ba\notin \langle \rho \rangle$;
\item For any arrow $b$, there is at most one arrow $c$ with $t(c)=s(b)$ and $cb\in \langle \rho \rangle$ and at most one arrow $a$ with $s(b)= t(a)$ and $ab\notin \langle \rho \rangle$;
\item $\rho$ generates an admissible ideal of $KQ$.
\end{enumerate}
\end{definition}

If $(Q, \rho)$ satisfies the first three conditions, then say that $KQ/\langle \rho \rangle$ is locally gentle. 

The class of gentle algebras is a subclass of another path algebra known as special biserial algebras, where the relations are not necessarily monomial. The next theorem by Kuber, Srivastava, and Sinha determines all possible stable ranks for special biserial algebras \cite{SVA23}.
\begin{theorem}\label{stablesba}
For any special biserial algebra $\Lambda$ with at least one band, $\omega \leq \mathrm{st}(\Lambda) < \omega^2$.    
\end{theorem}

\begin{definition}
A Skew-gentle algebra $\bar\Lambda$ is a bound quiver algebra $KQ/\langle \rho \rangle$ satisfying the following conditions:
\begin{enumerate}
\item $Q_1= Q'_1\cup S$ where, $S$ is a set of special loops;
\item $\rho= \langle \rho'\cup \{f^2-f\mid f\in S\}\rangle$;
\item $KQ'/\langle \rho' \rangle$ is a (locally) gentle algebra where $Q'=(Q_0, Q'_1)$;
\item If $f\in S$ then $x= s(f)= t(f)$ is the start or the end of exactly one arrow in $Q'_1$ and, if there is an arrow $\alpha\in Q'_1$ with $t(\alpha)=x$ and an arrow $\beta\in Q'_1$ with $s(\beta)=x$, then $\alpha\beta\in \rho'$. Moreover, there is no other loop at the vertex $x$.
\end{enumerate}
\end{definition}

Skew-gentle algebras are also discovered as the skew-group algebras of gentle algebras equipped with a certain $\mathbb{Z}_2$-action when the characteristic of $K$ is different from $2$. Therefore, as an application of Theorem \ref{stable}, we state the following corollary.

\begin{corollary}\label{SRA}
For a skew gentle algebra $\Lambda$ with at least one band, we have $\omega \leq \mathrm{st}(\Lambda) < \omega^2$.    
\end{corollary}  

Since a skew gentle algebra is a tame algebra \cite{G}, the above corollary also supports a conjecture given in \cite{SVA23} that claims that the statement about the stable ranks in the above corollary holds for any tame algebra that is not of finite representation type.

\begin{section}{Irreducible morphisms and Almost split sequences in $\mathrm{mod}\mbox{-}\Lambda G$}
In this section, we discuss irreducibility on morphisms in a skew group algebra. We produce an example to show that irreducibility is not preserved under $F_\lambda$ in general. Using the semi-dense property of $F_\lambda$, we provide the complete description of the irreducible morphisms and finally, we end the section describing the almost split sequences in $\mathrm{mod}\mbox{-}\Lambda G$.

\subsection{Irreducible morphisms in $\mathrm{mod}\mbox{-}\Lambda G$} 
Let $f: M\to N$ be a morphism in $\mathrm{mod}\mbox{-}\Lambda$. Recall that $f$ is irreducible if $f$ is neither a section nor a retraction, and every factorization $f= gh$ implies that $g$ is a section or $h$ is a retraction. If $M,N$ are indecomposable, then $f$ is irreducible if and only if it has a non-zero image in $\mathrm{irr}_\Lambda(M,N)= \mathrm{rad}_\Lambda(M,N)/\mathrm{rad}^2_\Lambda(M,N)$ where, $\mathrm{rad}_\Lambda(M,N)$ denotes the k-space of morphisms in the Jacobson radical of $\Lambda$.

It is well-known that if $f:M\to N$ belongs to $\mathrm{rad}(M,N) \setminus \mathrm{rad}^{2}(M,N)$ with either $M$ or $N$ indecomposable then $f$ is irreducible. The next result is a consequence of Theorem \ref{rank}. 

\begin{corollary}\label{irrsep}
Let $f:M\to N$ be an irreducible morphism in $\mathrm{mod}\mbox{-}\Lambda$ with $M,N \in \mbox{ind}(\Lambda)$. If $G_{M}\neq G$ or $G_{N}\neq G$, then $F_{\lambda}(f):F_{\lambda}(M) \to F_{\lambda}(N)$ is irreducible.
\end{corollary} 

\begin{proof}
By Theorem \ref{rank}, we know that $F_{\lambda}$ preserves powers of the radical. Hence, if $f:M\to N$ an irreducible morphism then $f\in \mathrm{rad}(M,N) \setminus \mathrm{rad}^{2}(M,N)$ then $F_{\lambda}(f) \in  \mathrm{rad}(F_\lambda(M),F_\lambda(N)) \setminus \mathrm{rad}^{2}(F_\lambda(M),F_\lambda(N))$. Since $F_\lambda(M)$ or $F_\lambda(N)$ are indecomposable by proposition \ref{Ind}, this implies that $F_{\lambda}(f)$ is irreducible.
\end{proof}

\begin{rmk}\label{irrcomp}
From the proof of Theorem \ref{rank}, it is clear that in general, if we consider an irreducible morphism $f:M\to N$ where both $M$ and $N$ are indecomposable, all non-zero entries of $F_\lambda(f)$ when viewed as a matrix via isomorphism are also irreducible morphisms.

Irreducibility may not be preserved if the hypothesis $G_{M}\neq G$ or $G_{N}\neq G$ is removed. If $G_{M}=G$ and $G_{N} = G$ then $F_\lambda(M)$ and $F_\lambda(N)$ are not indecomposable. Thus, if $f:M \to N$ is an irreducible morphisms with $M,N$ indecomposable modules, then $F_\lambda(f)$ can be interpreted as a matrix with all the diagonal entries irreducible morphisms but since both $F_\lambda(M)$ and $F_\lambda(N)$ are decomposable, $F_\lambda(f)$ may not be irreducible. 
\end{rmk}

 We will illustrate this in the next example.

\begin{example}\label{irrnp}
Consider the algebra $\Lambda$ and its skew group algebra $\Lambda \mathbb{Z}_2$ in Example \ref{ARQSKW} and an irreducible morphism $f: M \to N $ in $\mathrm{ind}\mbox{-}\Lambda$ with $G_{MN}=G$.

 Let  f:$\vcenter{\xymatrix@R=0.5mm{  1 \\ 2}} \to \vcenter{\xymatrix@R=0.5mm{2\\1\\2}} $ be an irreducible morphisms where the modules are given by their composition factors. 
 We have that  $F_{\lambda}(\vcenter{\xymatrix@R=0.5mm{ 1 \\ 2}}) = \vcenter{\xymatrix@R=0.5mm{ 3 \\ 2}} \bigoplus \vcenter{\xymatrix@R=0.5mm{5 \\ 4}}$ and $F_{\lambda}(\vcenter{\xymatrix@R=0.5mm{2\\1\\2}}) = \vcenter{\xymatrix@R=0.5mm{2\\3\\2}} \bigoplus \vcenter{\xymatrix@R=0.5mm{4\\5\\4}}$, thus 
 
 $F_{\lambda}(f) = \left( \begin{matrix} f_1 & 0\\ 0 & f_2 \end{matrix}\right):\vcenter{\xymatrix@R=0.5mm{ 3 \\ 2}} \bigoplus \vcenter{\xymatrix@R=0.5mm{5 \\ 4}} \to \vcenter{\xymatrix@R=0.5mm{2\\3\\2}} \bigoplus \vcenter{\xymatrix@R=0.5mm{4\\5\\4}} $  which is clearly not irreducible since can be factorized as a composition of irreducible morphisms, i. e.:  $$\left(\begin{array}{cc} f_1 & 0\\ 0 & f_2 \end{array}\right)=\left(\begin{array}{cc} f_1 & 0\\ 0 & 1 \end{array}\right).\left(\begin{array}{cc} 1 & 0\\ 0 & f_2 \end{array}\right).$$
\end{example}

The next result interplays between the irreducible morphisms in $\mathrm{mod}\mbox{-}\Lambda$ and $\mathrm{mod}\mbox{-}\Lambda G$. 
\begin{theorem}\label{irrars}
Assume that $G$ acts on an algebra $\Lambda$ and $\Lambda G$ is the associated skew group algebra. Then for any $M, N\in \mathrm{mod}\mbox{-}\Lambda$, the functor $F_\lambda: \mathrm{mod}\mbox{-}\Lambda \to \mathrm{mod}\mbox{-}\Lambda G$ induces the following isomorphisms of vector spaces:
$$\mathrm{irr}_{\Lambda G}(F_\lambda M,F_\lambda N)\approx \begin{cases}\bigoplus_{g\in G} \mathrm{irr}_{\Lambda} ({}^{g}M,N)&\mbox{ if } G_{M}\neq G,G_{N}= G;\\\bigoplus_{g\in G} \mathrm{irr}_{\Lambda}(M,{}^{g}N)&\mbox{ if }G_{N}\neq G, G_M=G.\end{cases}$$  
\end{theorem}
\begin{proof}
We only prove the first case since the others follow from the same argument. Let $\bar{f}$ a morphism in $\mathrm{Hom}_{\Lambda G}(F_\lambda M,F_\lambda N)$. Since $G_{M}\neq G$, we have that there exist $f \in \mathrm{Hom}_{\Lambda} (M,N)$ such that $F_\lambda (f)= \bar{f}$. Moreover, $\bar{f}$ corresponds via the isomorphism in Theorem \ref{diag} to $({}^gf)_{g\in G}$. Observe that, if $f$ is irreducible, then so is ${}^gf$. Thus $({}^gf)_{g\in G} \in \bigoplus_{g\in G} \mathrm{irr}_{\Lambda} ({}^gM,N)$. Now, if $f:M \to N$ is irreducible, then so is $F_\lambda (f)= \bar{f}$ by Corollary \ref{irrsep}.
\end{proof}

Unlike Galois covering, Galois semi-covering does not necessarily send an irreducible morphism to an irreducible morphism (see Proposition \ref{3.7} and Example \ref{irrnp}) when the group stabilizes both its source and target. The next result deals with this. 

\begin{proposition} \label{irrstab} 
Assume that $G$ acts on an algebra $\Lambda$ and $\Lambda G$ is the associated skew group algebra. Let $M,N \in \mbox{ind}(\Lambda G)$ such that $\hat{G}_{M}\neq \hat{G} $ and $\hat{G}_N \neq \hat{G}$. If $f \in  \mathrm{irr}_{\Lambda G}(M,N)$ then there exists a $f_1 \in  \mathrm{irr}_{\Lambda}(M_1,N_1)$ with $M_1,N_1 \in \mbox{ind}(\Lambda)$ such that ${}^gf$ for all $g\in G$ are the diagonal entries of the diagonal matrix $F_\lambda(f_1)$.
\end{proposition}

\begin{proof}
Consider the irreducible morphism $f: M \to N$, applying $G_\lambda$, we get  $G_\lambda(f):G_\lambda(M) \to G_\lambda(N)$. Observe that by hypothesis $\hat{G}_{M} \neq \hat{G}$ and $\hat{G}_{N} \neq \hat{G}$. Hence $G_\lambda M=M_1$ and $G_\lambda N=N_1$ are indecomposable by Proposition \ref{Ind}. Then we get $G_\lambda(f)=f_1:M_1\to N_1$.
    
Therefore, by Theorem \ref{diag}, we have $$F_\lambda(f_1)= \left(\begin{matrix} f & 0 & 0 & 0 \\ 0 &  {}^{g_1} f & 0 & 0\\ \cdots  &  \cdots & \cdots & \cdots  \\ 0 & 0 & 0 & {}^{g_n}f \end{matrix}\right)$$

This finishes the proof. \end{proof}

\begin{proposition} \label{irrstab} 
Assume that $G$ acts on an algebra $\Lambda$ and $\Lambda G$ is the associated skew group algebra. Let $M,N \in \mbox{ind}(\Lambda G)$ such that $\hat{G}_{MN}= \hat{G}$. If $f \in  \mathrm{irr}_{\Lambda G}(M,N)$ then there exists a $f_1 \in  \mathrm{irr}_{\Lambda}(M_1,N_1)$ with $M_1,N_1 \in \mbox{ind}(\Lambda)$ such that $f=F_\lambda(f_1)$.
\end{proposition}

\begin{proof}
Consider the irreducible morphism $f: M \to N$, applying $G_\lambda$, we get  $G_\lambda(f):G_\lambda(M) \to G_\lambda(N)$. We get $G_\lambda(f):\oplus_g {}^gM_1\to \oplus_g {}^gN_1$ where $M_1$ and $N_1$ satisfy $F_\lambda(M_1)=M$ and $F_\lambda(N_1)=N$. Therefore, by proposition \ref{3.7}, we have $$G_\lambda(f)= \left(\begin{matrix} f_1 & 0 & 0 & 0 \\ 0 &  {}^{g_1} f_1 & 0 & 0\\ \cdots  &  \cdots & \cdots & \cdots  \\ 0 & 0 & 0 & {}^{g_n}f_1 \end{matrix}\right)$$

\noindent where $f_1:M_1 \to N_1 $ is an irreducible morphism.  Hence $F_\lambda(f_1)=f$. 
\end{proof}

\subsection{Almost split sequences in $\mathrm{mod}\mbox{-}\Lambda G$}
Almost split sequences, also known as Auslander-Reiten (in short, A-R) sequences, are uniquely determined by their end modules (for details, see \cite{ARS}, Pg. 136). This section starts with a similar result that says that the stabilizer of an A-R sequence is also determined by the stabilizer of its end modules. For an almost split sequence $\mathcal{E}:0\to M\to N\to T\to 0$, define its stabilizer as $G_\mathcal{E}:=\{g\in G: g\mathcal{E}=\mathcal{E}\}.$


\begin{lemma}\label{ARS}
Suppose $G$ acts on an algebra $\Lambda$ with $\Lambda G$ as its skew group algebra. If $\mathcal{E}:0\to M\to N\to T\to 0$ is an A-R sequence in $\Lambda$ then the following are equivalent:
$(1)\mbox{ }G_\mathcal{E}=G, \mbox{ } (2)\mbox{ }G_M=G, \mbox{ }(3)\mbox{ }G_T=G.$

In any of these equivalent conditions, $N$ is of the form $N=\sum_{i\in I}N_i\oplus\sum_{j\in J}\sum_{k=1}^n {}^{g_k}N_j$ for (possibly empty) index sets $I$ and $J$ such that for each $i\in I, j\in J$ we have $G_{N_i}=G$ and $G_{N_j}\neq G$ respectively. 
\end{lemma}

\begin{proof}
Assume that $G_M= G$. If $G_N \neq G$ then there is an irreducible morphism from $M \to {}^{g_k}N$ for each $g_k\in G$, which is equivalent to the existence of another almost split sequence starting with $M$, a contradiction. Thus $G_N= G$. A similar argument ensures that, $G_T= G$ and hence, $G_\mathcal{E}=G$.

In any case, if $N_t$ is a direct summand of $N$ with $G_{N_t}\neq G$, then there is an irreducible morphism from $M \to {}^{g_k}N_t$ for each $g_k\in G$ and thus each ${}^{g_k}N_t$ is a direct summand of $N$, which ensures the form of $N$.
\end{proof}



\begin{theorem}\label{ARS}
Suppose $\mathcal{E}$ is an almost split sequence in $\Lambda$. Then the associated almost split sequence(s) in $\Lambda G$ have the following forms:
\begin{enumerate}
\item[$G_{\mathcal{E}}=G$] Here, $F_\lambda(\mathcal{E})=\uwithtext{Z}_{k=1}^n \bar{\mathcal{E}}^k$ where $\bar{\mathcal E}^k:=0\to {}^{\hat{g}_k} \bar{M}\to \sum_{j\in J}\bar{N_j}\oplus\sum_{i\in I}{}^{\hat{g}_k} \bar{N_i} \to {}^{\hat{g}_k} \bar{T}\to 0$ are the associated almost split sequences in $\Lambda G$ for each $k$ being glued via $Z:=\sum_{j\in J}\bar{N_j}$; In particular, if $Z=0$ then $F_\lambda(\mathcal{E})=\bigoplus_{k=1}^n \bar{\mathcal{E}}^k$. 

\item[$G_{\mathcal{E}}\neq G$] Here, $F_\lambda(\mathcal{E})=\bar{\mathcal{E}}$ where $\bar{\mathcal E}:=0\to \bar{M}\to \sum_{j\in J}\bar{N_j}\oplus\sum_{k=1}^n\sum_{i\in I}{}^{\hat{g}_k} \bar{N_i} \to \bar{T}\to 0$ are the associated almost split sequence in $\Lambda G$;
\end{enumerate}
such that for each $i\in I, j\in J$ we have $G_{\bar {N_i}}\neq G$ and $G_{\bar{N_j}}= G$ respectively.
\end{theorem}
\begin{proof}
Since $F_\lambda$ is exact, we have $F_\lambda(\mathcal E)$ is also an exact sequence. 

\noindent{\underline{$G_{\mathcal{E}}=G$:}} In this case, $G_{MT}=G$ by Lemma \ref{ARS}. Moreover, Proposition \ref{rem} says that $F_\lambda (M)=\bigoplus_{\hat{g}\in \hat{G}} {}^{\hat{g}} \bar M$, $F_\lambda (T)=\bigoplus_{\hat{g}\in \hat{G}} {}^{\hat{g}} \bar T$ for some $\bar M, \bar T \in \mathrm{Ind}\mbox{-}\Lambda G$ and  $F_\lambda (N)=\sum_{j\in J}\bar{N_j}\oplus\sum_{k=1}^n\sum_{i\in I}{}^{\hat{g}_k} \bar{N_i}$ where, for each $i\in I, j\in J$ we have $\bar N_i, \bar N_j \in \mathrm{Ind}\mbox{-}\Lambda G$ with $G_{\bar {N_i}}\neq G$ and $G_{\bar{N_j}}= G$ respectively. 

Moreover, in $\mathcal E$, there are irreducible morphisms from $M \to N_i$ and $M \to \sum_{k=1}^n {}^{g_k}N_j$ for each $i\in I, j\in J$. 
Since, $G_{MN_i}=G$ and $G_{N_j}\neq G$, there are irreducible morphisms from ${}^{\hat{g}_k}\bar M$ to ${}^{\hat{g}_k} \bar N_i$ and $\bar N_j$ for ${}^{\hat{g}} \bar M, {}^{\hat{g}_k} \bar N_i, \bar N_j \in \mathrm{Ind}\mbox{-}\Lambda G$ for each $i\in I, j\in J$ by Remark \ref{irrcomp}. There is no more irreducible morphism from ${}^{\hat{g}_k}\bar M$ in $\Lambda G$, otherwise it would produce another irreducible morphism from $M$ in $\Lambda$ under $G_\lambda$, which is not a part of $\mathcal{E}$, a contradiction. Hence, we get, ${}^{\hat{g}_k} \bar{M}\to \sum_{j\in J}\bar{N_j}\oplus\sum_{i\in I}{}^{\hat{g}_k} \bar{N_i}$ is irreducible for each $\hat{g}_k\in \hat{G}$. 

Similarly, we can show that $\sum_{j\in J}\bar{N_j}\oplus\sum_{i\in I}{}^{\hat{g}_k} \bar{N_i} \to {}^{\hat{g}_k} \bar{T}$ is irreducible for each $\hat{g}_k\in \hat{G}$. Therefore, exactness of $\bar{\mathcal E}^k$ concludes that it is also an almost split sequence. 

\noindent{\underline{$G_{\mathcal{E}}\neq G$:}} In this case, $G_{M}\neq G$ by Lemma \ref{ARS}. Moreover, Proposition \ref{rem} says that $G_{\bar{M}}=G$ and thus the result follows again by Lemma \ref{ARS}.
\end{proof}

We will illustrate this in the next example.

\begin{example}
Consider the algebra $\Lambda$ and its skew algebra $\bar{\Lambda}$ in example \ref{ARQSKW}. 
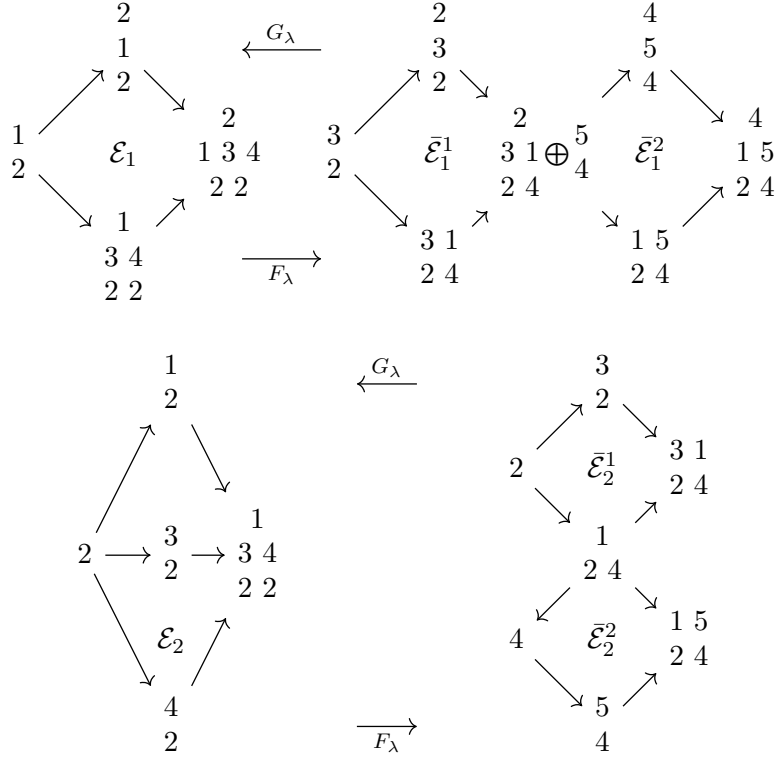
\[\begin{tikzcd}[sep={3.6em,between origins}]
                                                         & \begin{matrix} 2 \\  1\\ 2 \end{matrix} \arrow[rd]        & {}                                               & {} \arrow[l, "G_\lambda"']                               & \begin{matrix} 2 \\  3\\ 2 \end{matrix} \arrow[rd]  &                                                                                                                  & \begin{matrix} 4 \\  5\\ 4 \end{matrix} \arrow[rd]  &                                               \\
\begin{matrix} 1 \\ 2 \end{matrix} \arrow[ru] \arrow[rd] & \mathcal{E}_1                                               & \begin{matrix} 2 \\ 1\ 3\ 4 \\ 2\ 2 \end{matrix} & \begin{matrix} 3 \\ 2 \end{matrix} \arrow[ru] \arrow[rd] & \bar{\mathcal{E}}_1^1                                 & \begin{matrix} 2 \\ 3\ 1 \\ 2\ 4 \end{matrix} \bigoplus \begin{matrix} 5 \\ 4 \end{matrix} \arrow[ru] \arrow[rd] & \bar{\mathcal{E}}_1^2                                 & \begin{matrix} 4 \\ 1\ 5 \\ 2\ 4 \end{matrix} \\
                                                         & \begin{matrix} 1 \\  3\ 4 \\ 2\ 2 \end{matrix} \arrow[ru] & {} \arrow[r, "F_\lambda"']                       & {}                                                       & \begin{matrix} 3\ 1 \\ 2\ 4 \end{matrix} \arrow[ru] &                                                                                                                  & \begin{matrix} 1\ 5 \\ 2\ 4 \end{matrix} \arrow[ru] &                                              
\end{tikzcd}\]

\[
\begin{tikzcd}[sep={2.95em,between origins}]
                                    & \begin{matrix} 1 \\ 2 \end{matrix} \arrow[rdd] &                                                & {}                         & {} \arrow[l, "G_\lambda"'] &                         & \begin{matrix} 3 \\ 2 \end{matrix} \arrow[rd]                          &                                          \\
                                    &                                                &                                                &                            &                            & 2 \arrow[ru] \arrow[rd] & \bar{\mathcal{E}}_2^1                                                    & \begin{matrix} 3\ 1 \\ 2\ 4 \end{matrix} \\
2 \arrow[ruu] \arrow[rdd] \arrow[r] & \begin{matrix} 3 \\ 2 \end{matrix} \arrow[r]   & \begin{matrix} 1 \\  3\ 4 \\ 2\ 2 \end{matrix} &                            &                            &                         & \begin{matrix} 1 \\ 2\ 4 \end{matrix} \arrow[ru] \arrow[ld] \arrow[rd] &                                          \\
                                    & \mathcal{E}_2                                    &                                                &                            &                            & 4 \arrow[rd]            & \bar{\mathcal{E}}_2^2                                                    & \begin{matrix} 1\ 5 \\ 2\ 4 \end{matrix} \\
                                    & \begin{matrix} 4 \\ 2 \end{matrix} \arrow[ruu] &                                                & {} \arrow[r, "F_\lambda"'] & {}                         &                         & \begin{matrix} 5 \\ 4 \end{matrix} \arrow[ru]                          &                                         
\end{tikzcd}
\]
Here, we consider two A-R sequences $\mathcal{E}_i$ over $\Lambda$ with  $G_{\mathcal{E}_i}=G$ for $i\in \{1,2\}$. In the first example, $J=\phi$ and hence $F_\lambda(\mathcal{E}_1)$ is obtained by gluing $\bar{\mathcal{E}}_1^1$ and $\bar{\mathcal{E}}_1^2$ via the module $0$ i.e. $F_\lambda(\mathcal{E}_1)=\bar{\mathcal{E}}_1^1\uwithtext{0}\bar{\mathcal{E}}_1^2=\bar{\mathcal{E}}_1^1\bigoplus\bar{\mathcal{E}}_1^2$. Whereas, in the second example, we have $J=\{1,2\}$ and $N_1=\begin{matrix} 3 \\ 2 \end{matrix}$ and $N_2=\begin{matrix} 4 \\ 2 \end{matrix}$ and hence $F_\lambda(\mathcal{E}_2)$ is obtained by gluing $\bar{\mathcal{E}}_2^1$ and $\bar{\mathcal{E}}_2^2$ via the module $Z=\begin{matrix} 1 \\ 2\ 4 \end{matrix}$ i.e. $F_\lambda(\mathcal{E}_2)=\bar{\mathcal{E}}_2^1\uwithtext{Z}\bar{\mathcal{E}}_2^2$. On the other hand, in both the examples, for each $\bar{\mathcal{E}}_j^i$, $j\in \{1,2\}$, we have $G_{\bar{\mathcal{E}}_j^i}\neq G$ and applying $G_\lambda$, one can easily verify the desired result.   
\end{example}
\end{section}

\end{document}